\documentclass[preprint,11pt]{elsarticle}
\usepackage{amssymb,amsmath,amsthm,amsfonts,amscd}
\usepackage{graphicx,color}
\usepackage{color}
\usepackage{xcolor}
\usepackage{dsfont} 
\usepackage{cancel} 
\usepackage{inputenc} 
\usepackage{todonotes} 
\usepackage[shortlabels]{enumitem}
\journal{arXiv (preliminary version)}
\usepackage{hyperref}
\usepackage{comment}
\usepackage{float}
\newtheorem{theorem}{Theorem}[section]
\newtheorem{lemma}{Lemma}[section]
\newtheorem{proposition}{Proposition}[section]
\newtheorem{definition}{Definition}[section]
\newtheorem{remark}{Remark}[section]

\usepackage{tikz}
\usepackage{pgfplots}
\usetikzlibrary{patterns, arrows.meta}

\pgfplotsset{compat=1.18}
\usepackage{todonotes}

\begin{document}
	
	\begin{frontmatter}

\title{Stackelberg-Nash controllability for a multi-objective Stefan problem}
\author[a]{Thiago C. A de Carvalho}
\ead{thiago.carvalho@ufpi.edu.br}
\author[b,c]{Suerlan Silva}  
\ead{suerlansilva@id.uff.br}
\author[a]{Gilcenio R. de Sousa-Neto} 
\ead{gilceniorodrigues@ufpi.edu.br}
\author[a]{Franciane de B. Vieira}  
\ead{franciane@ufpi.edu.br}

\address[a]{Federal University of Piau\'{\i}, Department of Mathematics, PI, Brazil.}
\address[b]{Fluminense Federal University, Department of Mathematics, RJ, Brazil.}
\address[c]{Chair for Dynamics, Control, Machine Learning and Numerics,
Department of Mathematics, Friedrich-Alexander-Universit\"{a}t Erlangen--N\"{u}rnberg,
Cauerstrasse~11, 91058 Erlangen, Germany.}

\begin{abstract}
We investigate a hierarchical control problem for a one-dimensional Stefan system with localized distributed controls. The setting combines a Stackelberg strategy with a Nash equilibrium among  multiple followers, yielding a multi-objective free-boundary problem. The interaction between the hierarchical control and the moving interface results in a nonlinear optimality system, and we show that the original problem reduces to the null controllability of this optimality system. Under suitable geometric conditions on the control regions, we establish a local null controllability result. The proof relies on an observability inequality for a linearized system, obtained through Carleman estimates adapted to the presence of a moving boundary. These results constitute, to the best of our knowledge, the first treatment of a Stefan system within a Stackelberg–Nash framework.
\end{abstract}


\end{frontmatter}
\section{Introduction}
\setcounter{equation}{0}

The phenomena of melting and solidification arise in numerous natural and industrial contexts, such as ice thawing and metal casting. These processes involve a phase transition between solid and liquid states, and they can be described by the \textit{Stefan problem}. Its key feature is a moving interface that separates the phases, so that the spatial domain evolves in time.

Mathematically, the Stefan problem models the temperature field \(y(x,t)\) in a medium subject to heat diffusion and possible internal sources. The evolution of the temperature is governed by the equation
\begin{equation}\label{mainequation}
		\left\{
		\begin{aligned}
			&y_{t}-y_{xx}+a(x,t)y=f\mathds{1}_{\mathcal{O}}+ v_{1}\mathds{1}_{\mathcal{O}_{1}}+v_{2}\mathds{1}_{\mathcal{O}_{2}} &&\text{in } Q_{\ell},\\
			&y(0,t)=y(\ell(t),t)=0 &&\text{in }  (0,T),\\
			&y(x,0)=y_{0}(x) &&\text{in }  (0,\ell_{0}),
		\end{aligned}
		\right.
\end{equation}
defined in the time-dependent domain
\[
Q_{\ell}:=\{(x,t)\,:\ 0<x<\ell(t),\ 0<t<T\}.
\]

Here, \(y(x,t)\) represents the temperature, \(a(x,t)\) models heat absorption or reaction effects, and the functions \(f, v_{1}, v_{2}\) represent external heat sources acting in the subdomains \(\mathcal{O}, \mathcal{O}_1, \mathcal{O}_2\subset (0,\ell^{*})\), respectively. The characteristic function of a set \(A\) is denoted by \(\mathds{1}_A\). The endpoints \(x=0\) and \(x=\ell(t)\)  are kept at a reference temperature, set to zero for simplicity.
The spatial domain depends on time and extends from the fixed endpoint \(x=0\) up to the moving boundary \(x=\ell(t)\). The interface between solid and liquid phases is described by the moving boundary function $\ell\in\mathcal{F}(\ell^*,\ell_0, B)$, where
$$
\mathcal{F}(\ell^*,\ell_0, B)=\left\{\hat{\ell}\in C^1([0,T]):\ \hat{\ell}(0)=\ell_0,\ 0<\ell^{*}<\hat{\ell}(t)<B \text{ for all } t\in[0,T]\right\}.
$$
Here,  $T,\ell^{*},\ell_0,B>0$ are given constants.

The evolution of the moving boundary \(\ell(t)\) is determined by the \textit{Stefan condition}
\begin{equation}\label{stefancontion}
\ell'(t)=-\frac{1}{\beta}\,y_x(\ell(t),t), \quad t\in (0,T),
\end{equation}
where \(\beta>0\) is the latent heat coefficient. This condition establishes the relation between the motion of the phase interface and the heat flux at the moving boundary, evaluated as a one-sided limit from inside the domain \(0 < x < \ell(t)\). A negative value \(y_x(\ell(t),t) < 0\) indicates that heat flows into the material through \(x = \ell(t)\). In this case, the interface velocity satisfies \(\ell'(t) > 0\), meaning that the interface advances outward — the solid melts and the liquid region expands. Conversely, when \(y_x(\ell(t),t) > 0\), the heat flux is directed outward, that is, heat leaves the material through \(x = \ell(t)\). Consequently \(\ell'(t) < 0\) and the interface recedes, modeling solidification and a reduction of the liquid region.

The Stefan problem has been widely applied beyond classical melting and solidification processes. In engineering, it models free surface flows and fluid–solid interactions in casting, welding, and other thermal processes \cite{7,5,6}. In environmental and geotechnical contexts, it describes gas and fluid flow through porous media, where phase changes affect transport properties \cite{11,13}. In biology, Stefan-type models capture the growth and spread of tumors or other cell populations, with the moving interface representing the advancing front of living tissue \cite{15,14,FR99}. More recently, analogous formulations have been used to study the diffusion of information in social networks, where ``melting" and ``solidification" serve as metaphors for the adoption and stabilization of ideas \cite{LLW13}.

Our goal in this paper is to analyze the controllability of a multi-objective system based on the Stefan problem. Several works have studied controllability properties for Stefan problems. In \cite{fernandez2016controllability}, null controllability was established using Carleman inequalities combined with fixed-point arguments. In \cite{fernandez2017local}, the authors obtained local null controllability for semilinear problems via a fixed-point method, which was later applied to the viscous Burgers equation in \cite{fernandez2017localB}. A similar approach was used in \cite{fernandez2019local} to achieve local controllability for nonlinear parabolic equations, based on the Lusternik–Graves theorem. These results were further generalized in \cite{costa2023controllability} and \cite{wang2022local}, where the authors considered problems with local and nonlocal nonlinearities and quasilinear operators, respectively, using fixed-point techniques and H\"older regularity arguments. Another interesting question was answered in \cite{araujo2022remarks}, where the authors extended the analysis to a two-phase Stefan problem, accounting for the additional complexity of coupled phases. In a different direction, the first local exact controllability result was obtained in \cite{geshkovski2021controllability} for the viscous Burgers equation. The same approach was later applied in \cite{barcena2023exact} to parabolic equations using the Lusternik–Graves theorem. In a different context,  \cite{wang2022insensitizing} establishes the existence of insensitizing controls, while \cite{wang2023null} investigates the Stefan problem with multiplicative controls. Finally, the authors of   \cite{demarque2018local} studied the null controllability of one-phase Stefan problems in star-shaped domains.

Despite these significant contributions, most existing works focus on single-objective controllability and do not account for multiple criteria simultaneously. In contrast, the present work analyzes a multi-objective controllability problem for the Stefan system.

The multi-objective problem considered here involves three distinct objectives, each associated with one of the control functions $f$, $v_{1}$, and $v_{2}$. We assign the following tasks to each control:
\begin{itemize}
    \item \textbf{Task for $f$}: Drive the solution $y$ of \eqref{mainequation}–\eqref{stefancontion}, corresponding to $(v_{1},v_{2})$, to zero at the final time $T$, i.e., achieve null controllability of the system with respect to $f$.
    \item \textbf{Task for $v_{i}$}: Ensure that the solution $y$ associated with $f$ approximates a given target configuration $y_{i,d}$ as closely as possible in a prescribed subregion $\mathcal{O}_{i,d}$. This is formulated as the minimization of a suitable functional with respect to $v_{i}$.
  
\end{itemize}

These three tasks are in conflict and must therefore be addressed simultaneously. We must then choose an appropriate strategy. In this paper, we adopt the noncooperative Nash equilibrium concept \cite{nash} for the followers $v_{1},v_{2}$, combined with the hierarchical Stackelberg strategy \cite{stackelberg} in which $f$ acts as the leader and $v_{1},v_{2}$ as followers. This is the so-called Stackelberg–Nash strategy for multi-objective problems.

We now describe our problem precisely. Henceforth, we set the notation
$$X^{p}(0,T)=L^{p}(0,T;W^{2,p}(0,B)\cap W^{1,p}_{0}(0,B))\cap W^{1,p}(0,T;L^{p}(0,B)),$$
understood in the natural sense for functions defined on the non-cylindrical domain $Q_{\ell}$.
Let $\mathcal{O}_{1,d}, \mathcal{O}_{2,d} \subset (0,\ell^*)$ be open sets where the follower controls aim to track target functions. For each $i=1,2$, let $y_{i,d}=y_{i,d}(x,t) \in L^{2}(\mathcal{O}_{i,d} \times (0,T))$ be a given target function describing the desired configuration that the solution of \eqref{mainequation}–\eqref{stefancontion} is expected to approximate. We also introduce constants $\mu_i > 0$.

For each initial condition $y_0\in H_0^1(0,\ell_0)$, leader control $f\in L^2(\mathcal{O}\times(0,T))$, and moving boundary $\ell\in \mathcal{F}(\ell^*,\ell_0, B)$, we define the nonlinear and convex cost functionals
\[
J_i=J_i(y_0,f,\ell): L^2(\mathcal{O}_1\times(0,T)) \times L^2(\mathcal{O}_2\times(0,T)) \to \mathbb{R}, \quad i=1,2,
\]
by
\begin{equation}\label{secondfunctionals}
J_i(v_{1},v_{2})=\frac{1}{2} \iint_{\mathcal{O}_{i,d}\times(0,T)} |y-y_{i,d}|^2 \, dx\, dt + \frac{\mu_i}{2} \iint_{\mathcal{O}_i\times(0,T)} |v_{i}|^2 \, dx\, dt,
\end{equation}
where $y$ is the solution of \eqref{mainequation}–\eqref{stefancontion} corresponding to the data $(y_0,f,\ell,v_{1},v_{2})$.

\begin{remark}\label{remark_g1}
A solution of the Stefan problem \eqref{mainequation}--\eqref{stefancontion} is a pair $(y,\ell)$, consisting of the temperature function and the moving boundary. The Stefan problem does not necessarily admit a unique solution, since different functions $\ell(t)$ may satisfy the Stefan condition. However, once a specific boundary $\ell(t)$ is fixed (chosen from a pair $(y,\ell)$ that solves the Stefan problem), the temperature $y(x,t)$ is uniquely determined by system \eqref{mainequation} for given initial data. This ensures that the cost functionals $J_i$ in \eqref{secondfunctionals} are well-posed for each pair $(y,\ell)$ satisfying the Stefan problem. 
\end{remark}

The concept of Nash equilibrium is given below:
\begin{definition}
Given $(y_0,\ell,f)$ such that the functionals $J_i=J_i(y_0,f,\ell;\cdot)$, $i=1,2$, are well-posed, a pair $(v_{1},v_{2}) \in L^2(\mathcal{O}_1\times(0,T))\times L^2(\mathcal{O}_2\times(0,T))$ is called a \emph{Nash equilibrium} for the functionals $J_1, J_2$ if
\begin{equation}\label{minimousfunctional}
\left|
    \begin{array}{ll}\displaystyle
        J_1(v_{1},v_{2}) = \min_{\hat{v}_1\in L^2(\mathcal{O}_1 \times (0,T))} J_1(\hat v_{1}, v_{2}), \\ \noalign{\smallskip} 
        \displaystyle 
        J_2(v_{1},v_{2}) = \min_{\hat{v}_2 \in L^2(\mathcal{O}_2 \times (0,T))} J_2(v_{1},\hat v_{2}).
    \end{array}
    \right.
\end{equation}
\end{definition}
The hierarchy chosen for our problem designates $f$ as the leader. Therefore, our problem consists in finding a null control $f$ for the system \eqref{mainequation}–\eqref{stefancontion} subject to the condition that $(v_{1},v_{2})$ is a Nash equilibrium. Our main result is the following.

\begin{theorem}\label{MT}
    Let $\ell^*,\ell_0, B>0$, $y_0\in W_{0}^{1,4}(0,\ell_{0})$, and $y_{i,d}\in L^{4}(\mathcal{O}_{i,d} \times (0,T))$, $i=1,2$. Let us assume that one of the following two geometric configurations holds:
        \begin{equation}\label{GC_1}
        \begin{array}{lll}
             (G1)\ \ \mathcal{O}\cap\mathcal{O}_{i,d}\neq\emptyset \ for\ i=1,2, \ and\ \mathcal{O}_{1,d} = \mathcal{O}_{2,d};\\
           (G2)\ \ \mathcal{O}\cap\mathcal{O}_{i,d}\neq\emptyset \ for\ i=1,2, \ and\ \mathcal{O} \cap \mathcal{O}_{1,d} \neq\mathcal{O} \cap \mathcal{O}_{2,d}.
        \end{array}
        \end{equation}
    There exist constants $\varepsilon_0, \mu_0 > 0$ and a positive function $\rho(t)$ that blows up at $t=T$ and is non-constant only in an arbitrarily small neighborhood of $T$, such that when the conditions
    \begin{enumerate}[(1)]
        \item $\mu_1,\mu_2>\mu_0$,
        \item $\|y_0\|_{W_{0}^{1,4}(0,\ell_0)}^2+ \displaystyle\sum_{i=1}^{2} \|\rho y_{i,d}\|_{L^4(\mathcal{O}_{i,d}\times(0,T))}^2\leq \varepsilon_0,$
    \end{enumerate}
    are satisfied, then there exists a quintuple $$(y,\ell,f,v_{1},v_{2})\in X^{4}(0,T)\times\mathcal{F}(\ell^*,\ell_0, B)\times L^2(\mathcal{O}\times(0,T))\times L^2(\mathcal{O}_1\times(0,T))\times L^2(\mathcal{O}_2\times(0,T)),$$
    such that
 $(v_{1},v_{2})$ is the unique Nash equilibrium for the functionals $J_{1}(y_0,f,\ell),\ J_{2}(y_0,f,\ell)$, and $y$ is the solution of system \eqref{mainequation}–\eqref{stefancontion} associated with $(y_0,\ell,f,v_{1},v_{2})$ satisfying the null controllability condition $y(\cdot,T)=0$ in $L^2(0,\ell(T))$.
\end{theorem}

\begin{figure}[h]
   \centering
   \begin{tikzpicture}[scale=0.7]

\begin{scope}
\fill[gray!10] (0.6,0) rectangle (6.0,4.0);

\draw[thick, blue]
  (5.0,0) 
  .. controls (4.6,1.2) and (5.1,0.5) .. (5.4,2.0)
  .. controls (5.7,3.2) and (5.9,2.8) .. (6.0,4.0);
\node at (5.7,2.0) {$\ell$};

\draw[dashed, thick] (4.5,0) -- (4.5,4);


\draw (4.5,0) node[below] {$\ell_{*}$} node[circle,fill,inner sep=1pt] {};
\draw (5.0,0) node[below] {$\ell_0$} node[circle,fill,inner sep=1pt] {};

\fill[pattern=north west lines, pattern color=black, ] (1.0,0) rectangle (2.4,4.0);
\fill[white] (1.6,2.5) circle (0.40);
\node at (1.6,2.5) {$\scriptstyle\mathcal{O}_{1,d}$};

\fill[pattern=dots, pattern color=blue!90] (2.2,0) rectangle (3.0,4.0);
\draw[blue!90, thin] (2.2,0) rectangle (3.0,4.0);
\node at (2.6,-0.3) [text=blue!90] {$\mathcal{O}$};

\fill[pattern=north east lines, pattern color=black] (2.8,0) rectangle (4.2,4.0);
\fill[white] (3.6,1.0) circle (0.40);
\node at (3.6,1.0) {$\scriptstyle \mathcal{O}_{2,d}$};

\draw[->] (0.6,0) -- (6.2,0) node[right] {$x$};
\draw[->] (0.6,0) -- (0.6,4.2) node[above] {$t$};

\node at (3,-1.0) {\text{(a)}};

\draw[blue!90, very thick] (2.2,0) -- (3.0,0);
\end{scope}

\begin{scope}[xshift=6.5cm]
\fill[gray!10] (0.6,0) rectangle (6.0,4.0);

\draw[thick, blue]
  (5.0,0) 
  .. controls (4.6,1.2) and (5.1,0.5) .. (5.4,2.0)
  .. controls (5.7,3.2) and (5.9,2.8) .. (6.0,4.0);
\node at (5.7,2.0) {$\ell$};

\draw[dashed, thick] (4.5,0) -- (4.5,4);

\draw (4.5,0) node[below] {$\ell_{*}$} node[circle,fill,inner sep=1pt] {};
\draw (5.0,0) node[below] {$\ell_0$} node[circle,fill,inner sep=1pt] {};

\fill[pattern=north east lines, pattern color=black] (2.1,0) rectangle (3.7,4.0);


\fill[pattern=dots, pattern color=blue!90] (3.1,0) rectangle (3.9,4.0);
\draw[blue!90, thin] (3.1,0) rectangle (3.9,4.0);
\node at (3.4,-0.3) [text=blue!90] {$\mathcal{O}$};
 
\fill[pattern=north west lines, pattern color=black, ] (1.0,0) rectangle (3.4,4.0);
\fill[white] (1.57,2.5) circle (0.40);
\fill[white] (2.57,1.0) circle (0.40);
\node at (1.57,2.5) {$\scriptstyle\mathcal{O}_{1,d}$};
\node at (2.57,1.0) {$\scriptstyle\mathcal{O}_{2,d}$};



\draw[->] (0.6,0) -- (6.2,0) node[right] {$x$};
\draw[->] (0.6,0) -- (0.6,4.2) node[above] {$t$};

\node at (3,-1.0) {\text{(b)}};

\draw[blue!90, very thick] (3.1,0) -- (3.9,0);

\end{scope}

\begin{scope}[xshift=13.0cm]
\fill[gray!10] (0.6,0) rectangle (6.0,4.0);

\draw[thick, blue]
  (5.0,0) 
  .. controls (4.6,1.2) and (5.1,0.5) .. (5.4,2.0)
  .. controls (5.7,3.2) and (5.9,2.8) .. (6.0,4.0);
\node at (5.7,2.0) {$\ell$};

\draw[dashed, thick] (4.5,0) -- (4.5,4);

\draw (4.5,0) node[below] {$\ell_{*}$} node[circle,fill,inner sep=1pt] {};
\draw (5.0,0) node[below] {$\ell_0$} node[circle,fill,inner sep=1pt] {};

\fill[pattern=north west lines, pattern color=black, ] (1.6,0) rectangle (3.8,4.0);


\fill[pattern=dots, pattern color=blue!90] (1.2,0) rectangle (2.0,4.0);
\draw[blue!90, thin] (1.2,0) rectangle (2.0,4.0);
\node at (1.6,-0.3) [text=blue!90] {$\mathcal{O}$};

\fill[pattern=north east lines, pattern color=black] (1.6,0) rectangle (3.8,4.0);
\fill[white] (3.0,1.0) circle (0.40);
\fill[white] (3.0,3.0) circle (0.40);
\node at (3.0,1.0) {$\scriptstyle\mathcal{O}_{2,d}$};
\node at (3.0,3.0) {$\scriptstyle\mathcal{O}_{1,d}$};



\draw[->] (0.6,0) -- (6.2,0) node[right] {$x$};
\draw[->] (0.6,0) -- (0.6,4.2) node[above] {$t$};

\node at (3,-1.0) {\text{(c)}};
\draw[blue!90, very thick] (1.2,0) -- (2.0,0);

\end{scope}

\end{tikzpicture}
   \caption{The sets \(\mathcal{O}_{1,d}\), \(\mathcal{O}_{2,d}\), and \(\mathcal{O}\)
are represented by the black-hatched and blue-dotted regions,
respectively. Figure~(a) illustrates the first case of~\((G2)\), where
\(\mathcal{O}_{1,d} \cap \mathcal{O}_{2,d} = \emptyset\).
Figure~(b) illustrates the second case of~\((G2)\), where
\(\mathcal{O}_{1,d} \cap \mathcal{O}_{2,d} \neq \emptyset\) but
\(\mathcal{O}_{1,d} \neq \mathcal{O}_{2,d}\).
Figure~(c) illustrates the case where
\(\mathcal{O}_{1,d} = \mathcal{O}_{2,d}\).}
   \label{fig:changeofvariable0}
\end{figure}
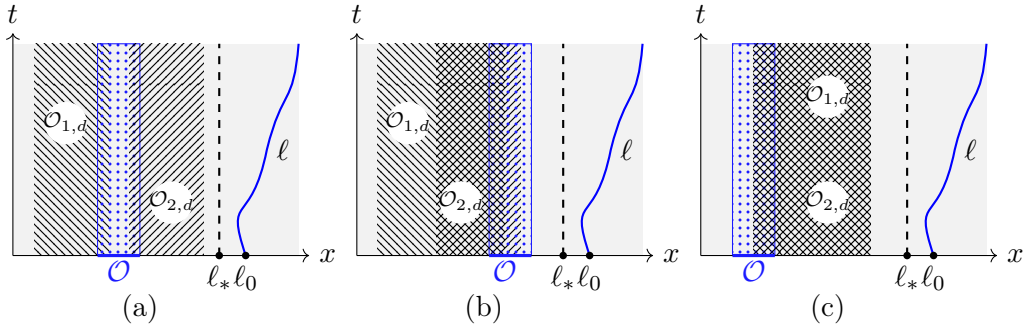


Stackelberg strategies in the context of partial differential equations were first introduced in \cite{lions1994some} and \cite{Lions1} by J.-L. Lions to study bi-objective control problems for the wave and heat equations. Later, \cite{Di-L} connected Stackelberg strategies with Nash equilibria \cite{nash}, adapting the framework to a setting with one leader and multiple followers. In this formulation, the leader seeks to steer the system toward an approximate controllability objective, while the followers minimize their respective cost functionals, resulting in a Stackelberg--Nash strategy.

In \cite{F-E-M}, the authors studied exact controllability for linear and semilinear parabolic equations with multi-objective controls under Stackelberg–Nash strategies, establishing conditions under which a Nash quasi-equilibrium becomes a true Nash equilibrium. These results were extended in \cite{F-E-M2} by relaxing the conditions on the observation domains of the followers. Further developments include the first results for degenerate parabolic equations under the same strategies in \cite{araruna2018stackelberg}, as well as a comprehensive study of hierarchical control for the semilinear case in \cite{Djomegne18112025}. More recently, the results concerning degenerate parabolic systems were generalized in \cite{LIMACO2026104513}. Within hierarchical strategies for coupled systems, the Stackelberg–Nash approach was investigated for parabolic systems in \cite{hernandez2018some,hernandez2016hierarchic}. In \cite{nina2021stackelberg}, this strategy was applied to a coupled quasilinear parabolic system with controls acting in the interior of the domain. In \cite{djomegne2023stackelberg}, the Stackelberg–Nash strategy was extended to coupled degenerate nonlinear parabolic equations. The existing literature has mainly focused on internal control. In contrast, the authors of \cite{araruna2020hierarchical} explored scenarios in which the controls act on the boundary, considering a mixed control setting involving both internal and boundary controls. In \cite{carreno2019stackelberg} and \cite{carreno2023stackelberg}, the Stackelberg–Nash strategy was applied to the Kuramoto–Sivashinsky equation, considering internal and distributed controls, respectively.

These results provide a framework for applying Stackelberg-Nash strategies to parabolic systems with multiple controls. Based on this, we now describe the strategy used to prove Theorem \ref{MT}, presenting the steps justified in the following sections.

\subsubsection*{\textbf{Step 1. Incorporation of the leader dependence}}

As introduced in \cite{F-E-M2}, when dealing with constrained problems, it is useful to incorporate the restriction — here, the dependence of the Nash equilibrium controls on the leader $f$ — into the system \eqref{mainequation}--\eqref{stefancontion}.

Let $(y,\ell)$ be a solution of the Stefan problem \eqref{mainequation}--\eqref{stefancontion} associated with the initial data $y_0$ and $f$. Let $(v_{1},v_{2})$ be a Nash equilibrium for the functionals $J_i=J_i(y_0,f,\ell;\cdot)$ corresponding to the pair $(y,\ell)$. Since each functional $J_i$ is convex, a pair $(v_{1},v_{2})$ is a Nash equilibrium for $J_1$ and $ J_2$ if and only if the following Gâteaux derivatives of $J_1$ and $J_2$ vanish:

\begin{equation}
	\begin{array}{lll}
    \displaystyle J_1'(v_{1},v_{2})\cdot (\hat{v}_1,0)=0,
    \quad\quad
    \displaystyle  J_2'(v_{1},v_{2})\cdot (0,\hat{v}_2)=0.
    \end{array}
\end{equation}
Since $y$ is the unique solution of \eqref{mainequation}--\eqref{stefancontion} corresponding to $(y_0,f,\ell,v_{1},v_{2})$, the following auxiliary system also has a unique solution:
\begin{equation}\label{phi_g1}
\left\{\begin{aligned}
			&-\phi^{i}_{t}-\phi^{i}_{xx}+a(x,t)\phi^{i}=(y-y_{i,d})\mathds{1}_{\mathcal{O}_{i,d}} &&\text{in}&& Q_{\ell},\\
			&\phi^{i}(0,t)=\phi^{i}(\ell(t),t)=0 &&\text{in} && (0,T), \\
			&\phi^{i}(x,T)=0 &&\text{in} &&(0,\ell(T)).
		\end{aligned}
		\right.
\end{equation}
Denoting $(\hat{V}_1,\hat{V}_2)=((\hat{v}_1,0),\ (0,\hat{v}_2))$, we can compute that
\begin{equation}
	\begin{array}{lll}
    \displaystyle J_i'(v_{1},v_{2})\cdot \hat{V}_i=\iint_{\mathcal{O}_{i}\times(0,T)}
    (\phi_i +\mu_{i}v_{i})\hat{v}_i \, dxdt,
     \end{array}
\end{equation}
where $\phi_i$ solves \eqref{phi_g1}.  Therefore, $(v_{1},v_{2})$ is a Nash equilibrium for $J_1,J_2$ associated with $(y_0,f,\ell)$, if and only if
\begin{equation}\label{char_g1}
(v_{1},v_{2})=\left(-\dfrac{1}{\mu_{1}}\phi^{1}\mathds{1}_{\mathcal{O}_1},-\dfrac{1}{\mu_{2}}\phi^{2}\mathds{1}_{\mathcal{O}_2}\right).
\end{equation}
The above analysis motivates the introduction of the nonlinear optimality system
 \begin{equation}\label{optimalsystem}
		\left\{
  \begin{array}{llllll}
	y_{t}-y_{xx}+a(x,t)y=f\mathds{1}_{\mathcal{O}}
    -\dfrac{1}{\mu_{1}}\phi^{1}\mathds{1}_{\mathcal{O}_{1}}-\dfrac{1}{\mu_{2}}\phi^{2}\mathds{1}_{\mathcal{O}_{2}}
    &&\text{in}&& Q_{\ell},
\\\noalign{\smallskip}
	-\phi^{i}_{t}-\phi^{i}_{xx}+a(x,t)\phi^{i}=(y-y_{i,d})\mathds{1}_{\mathcal{O}_{i,d}} &&\text{in}&& Q_{\ell},
\\\noalign{\smallskip}
    y(0,t)=y(\ell(t),t)=\phi^{i}(0,t)=\phi^{i}(\ell(t),t)=0 &&\text{in} && (0,T),
\\\noalign{\smallskip}
	y(\cdot,0)=y_{0}&&\text{in} && (0,\ell_{0}),
\\\noalign{\smallskip}
	\phi^{i}(\cdot,T)=0 &&\text{in} && (0,\ell(T)),
\\\noalign{\smallskip}
\ell'(\cdot)=-\dfrac{1}{\beta}y_{x}(\ell(\cdot),\cdot) &&\text{in} && (0,T).
		\end{array}
		\right.
\end{equation} 
As observed above, if a Nash equilibrium exists, the problem of finding a null control for the system \eqref{mainequation}--\eqref{stefancontion}, together with the determination of the Nash equilibrium, reduces to the null controllability of the optimality system \eqref{optimalsystem}. At this stage, the existence and uniqueness of the Nash equilibrium can be established by standard arguments. For instance, one may rely on the Lax--Milgram approach developed in \cite{F-E-M} to obtain the following existence result.

\begin{proposition}\label{Nashexistence}
     Let $(y,f,\ell)\in X^{2}(0,T)\times L^{2}(\mathcal{O}\times(0,T))\times\mathcal{F}(\ell^*,\ell_0, B)$ be a solution of \eqref{mainequation}–\eqref{stefancontion} corresponding to initial data $y_{0}\in H_0^1(0,\ell_0)$ and targets $y_{i,d} \in L^{2}(\mathcal{O}_{i,d} \times (0,T))$, $i=1,2$. There exists a constant $\mu_{01}>0$ such that if $\mu_1,\mu_2>\mu_{01}$, then there exists a unique Nash equilibrium $(v_{1},v_{2})$ for the functionals $J_1$ and $J_2$, associated with $(y,f,\ell)$. Moreover, 
     \begin{equation}\label{v12char_0}
        (v_{1},v_{2})=\left(-\dfrac{1}{\mu_{1}}\phi^{1}\mathds{1}_{\mathcal{O}_1},-\dfrac{1}{\mu_{2}}\phi^{2}\mathds{1}_{\mathcal{O}_2}\right).
    \end{equation}
    where $(\phi_1,\phi_2)$ composes the solution of \eqref{optimalsystem} associated with $(y_0,f,\ell,y_{1,d},y_{2,d})$.
\end{proposition}

In view of the preceding proposition, once the controllability of the system \eqref{optimalsystem} is established, the Nash equilibrium for the follower controls $(v_{1},v_{2})$ follows automatically. This is consistent with the hierarchical framework, in which the leader control $f$ is determined first, and the follower controls $v_{1}$ and $v_{2}$ are obtained subsequently. Hence, to prove Theorem \ref{MT}, it suffices to analyze the system \eqref{optimalsystem}, which already incorporates the Nash equilibrium problem for the followers. Consequently, Theorem \ref{MT} follows directly from the next result.

\begin{theorem}\label{MT2}
 Let $\ell^{*},\ell_0,B>0$, $y_0\in W_{0}^{1,4}(0,\ell_{0})$, and $y_{i,d}\in L^{4}(\mathcal{O}_{i,d} \times (0,T))$, $i=1,2$. Let us assume that one of the following two geometric configurations holds:
        \begin{equation}\label{GC_2}
        \begin{array}{lll}
             (G1)\ \ \mathcal{O}\cap\mathcal{O}_{i,d}\neq\emptyset \ for\ i=1,2 \ and\ \mathcal{O}_{1,d} = \mathcal{O}_{2,d};\\
           (G2)\ \ \mathcal{O}\cap\mathcal{O}_{i,d}\neq\emptyset \ for\ i=1,2 \ and\ \mathcal{O} \cap \mathcal{O}_{1,d} \neq\mathcal{O} \cap \mathcal{O}_{2,d}.
        \end{array}
        \end{equation}
   There exist constants $\varepsilon_0>0$, $\mu_0 > 0$ and a positive function $\rho(t)$ that blows up at $t=T$ and is nonconstant only in an arbitrarily small neighborhood of $T$, such that if the conditions
    \begin{enumerate}[(1)]
        \item $\mu_1,\mu_2>\mu_0$,
        \item $\displaystyle\|y_0\|_{W_{0}^{1,4}(0,\ell_0)}^2+\sum_{i=1}^{2} \|\rho y_{i,d}\|_{L^4(\mathcal{O}_{i,d}\times(0,T))}^2\leq \varepsilon_0,$
    \end{enumerate}
    are satisfied, then there exists a quintuple $$(y,f,\ell,\phi^1,\phi^2)\in X^{4}(0,T)\times L^2(\mathcal{O}\times(0,T))\times\mathcal{F}(\ell^*,\ell_0, B)\times L^2(Q_\ell)\times L^2(Q_\ell),$$
    associated with $y_0,y_{1,d},y_{2,d}$, that solves the Stefan problem \eqref{optimalsystem} and satisfies the null controllability condition $y(\cdot,T)=0$ in $L^2(0,\ell(T))$.
\end{theorem}

\subsubsection*{\textbf{Step 2. Controllability of the linearized optimal system}}

The proof of Theorem~\ref{MT2} consists in studying the linearized problem 
$\eqref{optimalsystem}_1$–$\eqref{optimalsystem}_5$, 
for which we establish the existence of suitable controls 
\(f\) corresponding to each moving boundary \(\ell(t)\) that make this linear problem approximately controllable. 
Once the approximate controllability of this linear system is established, we apply a fixed-point argument to deduce the null controllability stated in Theorem \ref{MT2} for the nonlinear Stefan problem \eqref{optimalsystem}.

The following theorem establishes the approximate controllability of $\eqref{optimalsystem}_1$–$\eqref{optimalsystem}_5$.

\begin{theorem}\label{aproxtheorem}
 Let $\ell^{*},\ell_0,B>0$, $y_0\in H_0^1(0,\ell_{0})$, $\varepsilon\in(0,1)$, and $y_{i,d}\in L^{2}(\mathcal{O}_{i,d} \times (0,T))$, $i=1,2$. Let one of the geometric configurations $(G1)$, $(G2)$ of Theorem \ref{MT2} hold.
    There exist a constant $\mu_{02}>0$ and a positive function $\rho(t)$ which blows up at $t=T$ and is nonconstant only in an arbitrarily small neighborhood of $T$, such that, if $\mu_1,\mu_2>\mu_{02}$ and the condition $(2)$ of Theorem \ref{MT} is satisfied, then there exists a function $f_{\varepsilon,\ell}\in L^2(\mathcal{O}\times(0,T))$ such that
    \begin{eqnarray}\label{fe}
     \|f_{\varepsilon,\ell}\|_{L^2(\mathcal{O}\times(0,T))}\leq \mathcal{C}\left(
            \|y_0\|_{H_0^1(0,\ell_0)}^2
            + \displaystyle\sum_{i=1}^{2} \|\rho y_{i,d}\|_{L^2(\mathcal{O}_{i,d}\times(0,T))}^2
                \right),
     \end{eqnarray}
     where the constant $\mathcal{C}>0$ is independent of $\varepsilon$ and $\ell\in \mathcal{F}(\ell^{*},\ell_0,B)$, and such that the solution $(y_{\varepsilon,\ell}, \phi^1_{\varepsilon,\ell}, \phi^2_{\varepsilon,\ell})$ of $\eqref{optimalsystem}_1$–$\eqref{optimalsystem}_5$ associated with $(y_0,f_{\varepsilon},\ell)$ satisfies
    \begin{eqnarray}\label{ye}
    \|y_{\varepsilon,\ell}(\cdot,T)\|_{L^2(0,\ell(T))}\leq \varepsilon.
    \end{eqnarray}
\end{theorem}

The proof of the above approximate controllability result relies on duality arguments that reduce the controllability property to an observability inequality for the adjoint system of  $\eqref{optimalsystem}_1$–$\eqref{optimalsystem}_5$. In our case, the adjoint system is given by

\begin{equation}\label{adjointoptimalsystem}
\left\{\begin{aligned}
&-\psi_{t}-\psi_{xx}+a(x,t)\psi=\gamma^{1}\mathds{1}_{\mathcal{O}_{1,d}}+\gamma^{2}\mathds{1}_{\mathcal{O}_{2,d}} &&\text{in}&& Q_{\ell},\\
&\gamma^{i}_{t}-\gamma^{i}_{xx}+a(x,t)\gamma^{i}=-\dfrac{1}{\mu_{i}}\psi\mathds{1}_{\mathcal{O}_{i}} &&\text{in}&& Q_{\ell},\\
&\gamma^{i}(0,t)=\gamma^{i}(\ell(t),t)=0,\,\psi(0,t)=\psi(\ell(t),t)=0 &&\text{in} && (0,T),\\
&\gamma^{i}(x,0)=0 &&\text{in} && (0,\ell_{0}),\\
&\psi(x,T)=\psi^{T}(x) &&\text{in} && (0,\ell(T)),
		\end{aligned}
		\right.
\end{equation} 
This system is connected to $\eqref{optimalsystem}_1$–$\eqref{optimalsystem}_5$ through the identity
\begin{equation}\label{connection}
\begin{array}{lll}
\displaystyle\int_{0}^{\ell(T)} y(x,T)\psi^T(x)dx
\\\noalign{\smallskip}\phantom{0000}
=
\displaystyle\int_{0}^{\ell_{0}} y_{0}(x)\,\psi(x,0)dx
+
\displaystyle\sum_{i=1}^{2}\iint_{\mathcal{O}_{i,d}\times(0,T)} y_{i,d}\gamma^{i}dxdt
+\displaystyle\iint_{\mathcal{O}\times(0,T)} f\psi dxdt,
\end{array}
\end{equation}
which motivates the observability inequality stated in the next theorem.

\begin{theorem}\label{obsth}  Let $\ell\in \mathcal{F}(\ell^*,\ell_0, B)$, $\psi^{T}\in L^{2}(0,\ell(T))$, and let us assume that one of the geometric configurations $(G1)$, $(G2)$ of Theorem \ref{MT2} holds. There exist a constant $\mu_{03}>0$ and a positive function $\rho(t)$ that blows up at $t=T$ and is non-constant only in an arbitrarily small neighborhood of $T$, such that if $\mu_1,\mu_2>\mu_{03}$ and the condition $(2)$ of Theorem \ref{MT} is satisfied, then the solution $(\psi,\gamma_1,\gamma_2)$ of \eqref{adjointoptimalsystem} associated with  $\ell$ and $\psi^{T}$ satisfies the inequality
\begin{equation}\label{obsinequalityintro}
    \int_{0}^{\ell_0}|\psi(x,0)|^{2}dx
    +
    \sum_{i=1}^{2}\iint_{Q_{\ell}}\rho^{-2}|\gamma^{i}|^{2}dxdt
        \leq
    C_0\iint_{\mathcal{O}\times(0,T)}|\psi|^{2}dxdt,
\end{equation}
where $C_0>0$ is a constant independent of $\ell\in \mathcal{F}(\ell^{*},\ell_0,B)$.
\end{theorem}

For the proof of the observability inequality, we employ Carleman estimates. 
Thus, our main goal is to derive a global estimate for the triple 
\((\psi,\gamma^{1},\gamma^{2})\) in terms of a single observation of \(\psi\) on 
\(\mathcal{O}\times(0,T)\). Our approach is motivated by \cite{F-E-M2}, where the authors 
do not use a single Fursikov weight but rather two different weight functions 
in order to handle both cases $(G1)$ and $(G2)$ in \eqref{GC_2}. The construction 
of suitable weight functions is nontrivial, and becomes even more delicate when the 
moving boundary \(\ell(t)\) is present. In \cite{F-E-M2}, the two weights satisfy special compatibility conditions that allow one to combine two weighted Carleman estimates. However, these weights are not directly applicable to free-boundary problems of Stefan type because they do not take into account the motion of the boundary. For free-boundary problems, a viable 
alternative is to adapt the weights introduced in \cite{fernandez2019local}. 
Our idea here is to build a new weight function that permits the manipulation and 
matching of the two Carleman inequalities in the spirit of the constructions in 
\cite{F-E-M2} and \cite{fernandez2019local}. The Fursikov-type weight that will be used in this work is introduced in \eqref{weightfunctions}, and its main properties are collected in Lemma~\ref{Furshikov-lemma}. For organizational reasons, we postpone to Appendix \ref{AppendixA} the proof of the lemma, in which the weight is obtained constructively. An advantage of this constructive approach is that the weight can be easily implemented in numerical schemes, including in the cylindrical case $\ell=const$., which may be useful for computational experiments.

\subsubsection*{\textbf{Step 3. Controllability of the nonlinear optimal system}}

Once the controllability of the linearized system has been established, we
proceed to prove the controllability of the nonlinear problem, that is,
Theorem~\ref{MT2}. To this end, we use a fixed-point argument to pass from the
linearized system to the nonlinear one, a procedure that is by now standard in
the literature. However, as pointed out in~\cite{fernandez2016controllability}
(see also~\cite{araujo2022remarks}), a major technical difficulty lies in
obtaining compactness for the operator that generates the fixed point. We follow \cite{fernandez2016controllability} to obtain compactness from a Hölder estimate for the normal derivative $y_x$. However, the essential difference between \cite{fernandez2016controllability} and our setting is that we are dealing with a coupled system rather than a single equation. Therefore, we must verify whether such a H\"older estimate holds for the solution $(y,\phi^1,\phi^2)$ of the linearized optimality system \(\eqref{optimalsystem}_1\)--\(\eqref{optimalsystem}_5\).
\\

We now outline the structure of the paper. In Section~\ref{carleman}, we derive a Carleman estimate for the adjoint system \eqref{adjointoptimalsystem}. This result is then used in Section~\ref{observability} to derive the observability inequality stated in Theorem \ref{obsth}. Section~\ref{approximatesec} is devoted to proving the approximate controllability stated in Theorem~\ref{aproxtheorem}. Finally, in Section~\ref{proofmainresult}, we prove Theorem~\ref{MT2}, which solves the multi-objective Stefan problem proposed in this work. We close with concluding remarks in Section~\ref{conclusion}, followed by an appendix that contains supporting technical material.

\section{Carleman estimate}\label{carleman}

In this section, we derive a Carleman estimate under each 
geometric configuration \eqref{GC_2} for the adjoint system 
\eqref{adjointoptimalsystem}, which we recall here:
\begin{equation*}
    \left\{
        \begin{array}{llllll}
            -\psi_{t}-\psi_{xx}+a(x,t)\psi=\gamma^{1}\mathds{1}_{\mathcal{O}_{1,d}}+\gamma^{2}\mathds{1}_{\mathcal{O}_{2,d}} 
                &&\text{in}&& Q_{\ell},
            \\\noalign{\smallskip}
            \gamma^{i}_{t}-\gamma^{i}_{xx}+a(x,t)\gamma^{i}=-\dfrac{1}{\mu_i}\psi\mathds{1}_{\mathcal{O}_{i}} 
                &&\text{in}&& Q_{\ell},
            \\\noalign{\smallskip}
            \psi(0,t)=\psi(\ell(t),t)=0,\;\gamma^{i}(0,t)=\gamma^{i}(\ell(t),t)=0 
                &&\text{in} && (0,T),
            \\\noalign{\smallskip}
            \psi(\cdot,T)=\psi^T &&\text{in} && (0,\ell(T)),
            \\\noalign{\smallskip}
            \gamma^{i}(\cdot,0)=0 &&\text{in} && (0,\ell_0).
        \end{array}
    \right.
\end{equation*}

The Fursikov weight functions to be used in the Carleman estimate for this system will be built from the auxiliary functions given in the following lemma.

\begin{lemma}\label{Furshikov-lemma} 
Let $\ell\in \mathcal{F}(\ell^*,\ell_0, B)$ and let $(\tilde{a},\tilde{b})\subset (0, \ell^*)$ be a set satisfying
$$
(\tilde{a},\tilde{b}) \cap \mathcal{O}_{i,d}\neq\emptyset,\quad i=1,2.
$$
%
\begin{enumerate}[label=(F\arabic*)]
    \item\label{F1} If $\mathcal{O}_{1,d}=\mathcal{O}_{2,d}$, then, given  $\omega_{0}\subset\subset(\tilde{a},\tilde{b})\cap\mathcal{O}_{1,d}\cap\mathcal{O}_{2,d}$, there exists an associated function $\eta^{*}_{0}\in C^2(\overline{Q_\ell})$ such that
\begin{equation*}
		\left\{
  \begin{array}{llllll}
	\eta^{*}_{0}>0 \ \text{in} \ Q_{\ell},
\\\noalign{\smallskip}
	\eta^{*}_{0}(0,t)=\eta^{*}_{0}(\ell(t),t)=0 \ \text{in} \ (0,T),
\\\noalign{\smallskip}
	|(\eta^*_0)_x|> c > 0 \ \text{in} \ \overline{Q}_{\ell}\backslash({\omega}_{0}\times(0,T)),
\\\noalign{\smallskip}
	\eta^{*}_{0}=1-\dfrac{x-\tilde{b}}{\ell(t)-\tilde{b}}, \ \forall x\in (\tilde{b},\ell(t)),\  \forall t\in[0,T].
        \end{array}
		\right.
\end{equation*}

    \item\label{F2} If $\mathcal{O}\cap\mathcal{O}_{1,d}\neq\mathcal{O}\cap\mathcal{O}_{2,d}$, then, given $\omega_{i}\subset\subset(\tilde{a},\tilde{b})\cap\mathcal{O}_{i,d}$, there exist associated functions $\eta_{i}^{*}\in C^2(\overline{Q_\ell})$,  $i=1,2$, such that $\|\eta^{*}_1\|_\infty = \|\eta^{*}_2\|_\infty$ and
    \begin{equation*}
		\left\{
  \begin{array}{llllll}
	\eta^{*}_{i}>0 \ \text{in} \ Q_{\ell},
\\\noalign{\smallskip}
	\eta^{*}_{i}(0,t)=\eta^{*}_{i}(\ell(t),t)=0 \ \text{on} \ (0,T),
\\\noalign{\smallskip}
	|(\eta^*_i)_x|> c > 0 \ \text{in} \ \overline{Q}_{\ell}\backslash({\omega}_{i}\times(0,T)),
\\\noalign{\smallskip}
	\eta^{*}_{1}=\eta^{*}_{2} \ \text{in} \ \overline{Q_{\ell}}\backslash((\tilde{a},\tilde{b})\times(0,T)),
\\\noalign{\smallskip}
	\eta^{*}_{i}=1-\dfrac{x-\tilde{b}}{\ell(t)-\tilde{b}}, \ \forall x\in (\tilde{b},\ell(t)),\ \forall t\in[0,T].
        \end{array}
		\right.
\end{equation*}
\end{enumerate}
\end{lemma}

The proof of this lemma is given in the Appendix \ref{AppendixA}. We now introduce the Fursikov weight functions
\begin{equation}\label{weightfunctions}
\sigma_{i}(x,t)=\dfrac{e^{4\lambda\|\eta_{i}\|_{\infty}}-e^{\lambda(2\|\eta_{i}\|_{\infty}+\eta_{i}(x,t))}}{t(T-t)} \ \ \text{and} \ \ \ \xi_{i}(x,t)=\dfrac{e^{\lambda(2\|\eta_{i}\|_{\infty}+\eta_{i}(x,t))}}{t(T-t)},  
\end{equation}
where $\eta_{i}(x,t)=\eta^{*}_{i}(x,t)+1$, $i=0,1,2$, and $(\eta_i^*,\omega_i)$ is furnished by Lemma \ref{Furshikov-lemma}. The following Carleman estimates are sufficient to establish the observability inequality.

\begin{theorem}\label{newcarleman} 
Let $\mathcal{O}_{i,d}\cap\mathcal{O}\neq\emptyset$, for $i=1,2$, and let us consider $\ell\in \mathcal{F}(\ell^*,\ell_0, B)$. Then, there exist constants $\mu_{00}, s_0,\lambda_0$ and $C>0$ such that, if $\mu_1,\mu_2>\mu_{00}$, the solution $(\psi,\gamma^{1},\gamma^{2})$ of \eqref{adjointoptimalsystem} corresponding to $\psi^{T}\in L^{2}(0,\ell(T))$ satisfies the following statements for all $s\geq s_0$ and $\lambda\geq\lambda_0$.
\begin{enumerate}[label=(\arabic*)]
    \item\label{(1)} If $\mathcal{O}_{1,d}=\mathcal{O}_{2,d}$, then 
    \begin{equation}
    \iint_{Q_{\ell}} e^{-2s\sigma_{0}}(\xi_{0})^{3}|\psi|^{2}dxdt 
    \leq
    C s\lambda\iint_{\mathcal{O}\times(0,T)} e^{-2s\sigma_{0}}(\xi_{0})^{4}|\psi|^{2}dxdt.
    \end{equation}
    \item\label{(2)} If $\mathcal{O}\cap\mathcal{O}_{1,d}\neq\mathcal{O}\cap\mathcal{O}_{2,d}$, then
      \begin{equation}
    \iint_{Q_{\ell}} e^{-2s\sigma_{1}}(\xi_{1})^{3}|\psi|^{2}dxdt 
    \leq
    C s\lambda\iint_{\mathcal{O}\times(0,T)} \left( e^{-2s\sigma_{1}}(\xi_{1})^{4} + e^{-2s\sigma_{2}}(\xi_{2})^{4}\right)|\psi|^{2} dxdt.
    \end{equation}
\end{enumerate}
\end{theorem}

The proof of Theorem \ref{newcarleman} relies on the manipulation of a standard Carleman inequality for the generic system
\begin{equation}\label{u-carlemansystem}
\left\{\begin{aligned}
			&-u_{t}-u_{xx}+a(x,t)u=F+G_x &&\text{in}&& Q_{\ell},\\
			&u(0,t)=u(\ell(t),t)=0 &&\text{in} &&(0,T), \\
			&u(x,T)=u^{T}(x) &&\text{in} &&(0,\ell(T)). \\
		\end{aligned}
		\right.
\end{equation}
More precisely, denoting
\begin{equation*}
\begin{array}{lll}
I_{m}^{i}(u)
:= s^{m-2}\lambda^{m-1}\displaystyle\iint_{Q_{\ell}}e^{-2s\sigma_{i}}(\xi_{i})^{m-2}|u_{x}|^{2}dxdt+s^{m}\lambda^{m+1}\iint_{Q_{\ell}}e^{-2s\sigma_{i}}(\xi_{i})^{m}|u|^{2}dxdt,
\end{array}
\end{equation*}
we will use the following result.
\begin{lemma}\label{standardcarleman_}
Let $\sigma_i, \xi_i$ be the weight functions given in \eqref{weightfunctions}, associated with the open set $\omega_i$. For all $m\in\mathbb{N}$ 
there exist constants $\lambda_m, s_m, C_m>0$ such that, if $s\geq s_m$ and $\lambda\geq\lambda_m$, the solution $u$ of \eqref{u-carlemansystem} associated with $u^T\in L^2(0,\ell(T))$ satisfies
\begin{equation}
\begin{array}{lll}
I_{m}^{i}(u)
\leq
C_m\left(s^{m}\lambda^{m+1}\displaystyle\iint_{\omega_{i}\times(0,T)}e^{-2s\sigma_{i}}(\xi_{i})^{m}|u|^{2}dxdt\right.
\\\noalign{\smallskip}\phantom{I_{m}^{i}(u)=}
+s^{m-3}\lambda^{m-3}\displaystyle\iint_{Q_{\ell}}e^{-2s\sigma_{i}}(\xi_{i})^{m-3}|F|^{2}dxdt
\\\noalign{\smallskip}\phantom{I_{m}^{i}(u)=}
+\left. s^{m-1}\lambda^{m-1}\displaystyle\iint_{Q_{\ell}}e^{-2s\sigma_{i}}(\xi_{i})^{m-1}\left|G\right|^{2}dxdt\right), \ \ i=0,1,2,\ \ m\in\mathbb{N}.
\end{array}
\end{equation}
\end{lemma}

The proof of this lemma is standard and can be obtained by following the ideas of \cite{imanuvilov2003carleman} and \cite{fernandez2016controllability}. 

Before proving Theorem \ref{newcarleman}, we present a lemma containing a computation common to both cases.

\begin{lemma}\label{lemma-carlema1} Let $(\psi,\gamma^{1},\gamma^{2})$ be a solution to \eqref{adjointoptimalsystem}, $\theta\in C^{2}(\mathbb{R};[0,1])$ a function such that $\theta=0$ in $\mathbb{R}\backslash\mathcal{O}$, and $m_0,m_1,m_2>0$ constants. Then, there exists a constant $C=C(m_0,m_1,m_2)>0$ such that
\begin{equation}
\begin{array}{lll}
m_{0}\lambda\displaystyle\iint_{\mathcal{O}\times(0,T)}e^{-2s\sigma_{i}}\theta(m_{1}\gamma^{1}+m_{2}\gamma^{2})(-\psi_{t}-\psi_{xx}-a(x,t)\psi)dxdt
\\\noalign{\smallskip}\phantom{0000}
\leq \dfrac{1}{2}I^{i}_{0}(m_{1}\gamma^{1}+m_{2}\gamma^{2})+C s^{4}\lambda^{5}\displaystyle\iint_{\mathcal{O}\times(0,T)}e^{-2s\sigma_{i}}(\xi_{i})^{4}|\psi|^{2}dxdt,\quad i=1,2,
\end{array}
\end{equation}
for all $s$ and $\lambda$ sufficiently large.
\end{lemma}
\begin{proof}
Since
\begin{equation}
\begin{array}{c}
\left| (e^{-2s\sigma_i}\theta)_{t} \right| 
+ \left| (e^{-2s\sigma_i}\theta)_{xx} \right|
\leq C s^{2} \lambda^{2} e^{-2s\sigma_i} (\xi_i)^{2},
\\\noalign{\smallskip}
\left| (e^{-2s\sigma_i}\theta)_x \right|
\leq C s \lambda e^{-2s\sigma_i} \xi_{i},
\end{array}
\end{equation}
then, denoting $h=m_1\gamma_1+m_2\gamma_2$, integration by parts yields
\begin{equation}
\begin{array}{lll}
\displaystyle\iint_{\mathcal{O}\times(0,T)} e^{-2s\sigma_{i}} \theta h (-\psi_{t} - \psi_{xx} + a \psi) \, dx \, dt
\\\noalign{\smallskip}\phantom{0000}
= - \displaystyle\iint_{\mathcal{O} \times (0,T)} e^{-2s\sigma_{i}} \theta 
\left( \dfrac{m_{1}}{\mu_{1}} \mathds{1}_{\mathcal{O}_{1}} 
     + \dfrac{m_{2}}{\mu_{2}} \mathds{1}_{\mathcal{O}_{2}} \right) 
     |\psi|^{2} \, dx \, dt
\\\noalign{\smallskip}\phantom{0000=}
+ \displaystyle\iint_{Q_\ell} 
\left[ (e^{-2s\sigma_{i}}\theta)_t h 
- (e^{-2s\sigma_{i}}\theta)_{xx} h 
- 2 (e^{-2s\sigma_{i}}\theta)_x h_x \right] 
\psi \, dx \, dt 
\\\noalign{\smallskip}\phantom{0000}
\leq
C \displaystyle\iint_{\mathcal{O} \times (0,T)} e^{-2s\sigma_{i}} 
\left( |\psi|^{2} + s^{2} \lambda^{2} (\xi_{i})^{2} 
|h| \, |\psi| + s \lambda \xi_{i} |h_x| \, |\psi|\right)dxdt
\\\noalign{\smallskip}\phantom{0000}
\leq
\dfrac{1}{m_{0} \lambda} \left( \dfrac{1}{2} I^{i}_{0}(h) 
+ C s^{4} \lambda^{5} \displaystyle\iint_{\mathcal{O} \times (0,T)} 
e^{-2s\sigma_{i}} (\xi_{i})^{4} |\psi|^{2} \, dx \, dt \right),
\end{array}
\end{equation}
for $s$ and $\lambda$ large enough.
\end{proof}
We are now ready to prove Theorem \ref{newcarleman}.
\subsection{Proof of Theorem \ref{newcarleman}: case \ref{(1)}}

Let $(\psi,\gamma_1,\gamma_2)$ be the solution of system \eqref{adjointoptimalsystem}.
Since $\mathcal{O}\cap\mathcal{O}_{i,d}\neq\emptyset$ and, in this case, $\mathcal{O}_{1,d}=\mathcal{O}_{2,d}$, we can take a set $\omega_0\subset\subset \mathcal{O}\cap\mathcal{O}_{1,d}\cap\mathcal{O}_{2,d}$ and consider its associated Fursikov weight $\eta_0^*$ furnished by Lemma \ref{Furshikov-lemma}.

Applying the Carleman estimate of Lemma \ref{standardcarleman_} with $m=3$ and $i=0$ for $\psi$, we obtain
\begin{equation}\label{T2.1G_1}
\begin{array}{lll}
I_{3}^{0}(\psi)
    \leq
        C\left(s^{3}\lambda^{4}\displaystyle\iint_{\omega_{0}\times(0,T)}e^{-2s\sigma_{0}}(\xi_{0})^{3}|\psi|^{2}dxdt+\iint_{Q_{\ell}}e^{-2s\sigma_{0}}|h|^{2}dxdt\right), 
\\\noalign{\smallskip}\phantom{I_{3}^{0}(\psi)}
    \leq
        C s^{3}\lambda^{4}\displaystyle\iint_{\omega_{0}\times(0,T)}e^{-2s\sigma_{0}}(\xi_{0})^{3}|\psi|^{2}dxdt
        + \dfrac{1}{2}I^{0}_{0}(h),  
\end{array}
\end{equation}
for $s$ and $\lambda$ large enough, where $h=\gamma^{1}+\gamma^{2}$.
Applying Lemma \ref{standardcarleman_} to $h=\gamma^{1}+\gamma^{2}$ with $m=0$ and $i=0$, we obtain
\begin{equation}\label{T2.1G_2}
\begin{array}{lll}
I^{0}_{0}(h)
    \leq
        C\left(\lambda\displaystyle\iint_{\omega_{0}\times(0,T)}e^{-2s\sigma_{0}}|h|^{2}dxdt+Cs^{-3}\lambda^{-3}\iint_{Q_{\ell}}e^{-2s\sigma_{0}}(\xi_{0})^{-3}|\psi|^{2}dxdt\right)
\\\noalign{\smallskip}\phantom{I^{0}_{0}(h)}
    \leq
        C\lambda\displaystyle\iint_{\omega_{0}\times(0,T)}e^{-2s\sigma_{0}}|h|^{2}dxdt
        + \dfrac{1}{2}I^{0}_{3}(\psi),  
\end{array}
\end{equation}
for $s$ and $\lambda$ large enough. Combining \eqref{T2.1G_1} and \eqref{T2.1G_2} and recalling that $\omega_0\subset\mathcal{O}$, we obtain
\begin{equation}\label{T2.1G_3}
\begin{array}{lll}
I_{3}^{0}(\psi) + I^{0}_{0}(h)
    \leq
        Cs^{3}\lambda^{4}\displaystyle\iint_{\mathcal{O}\times(0,T)}e^{-2s\sigma_{0}}(\xi_{0})^{3}|\psi|^{2}dxdt
        +
        C \lambda\displaystyle\iint_{\omega_{0}\times(0,T)}e^{-2s\sigma_{0}}|h|^{2}dxdt.
\end{array}
\end{equation}
The localized integral over $\omega_0$ in \eqref{T2.1G_3} can be absorbed. Indeed, since $\omega_0\subset\subset \mathcal{O}_{1,d}=\mathcal{O}_{2,d}$,
$$
h\mathds{1}_{\omega_0}
    =
        (\gamma_1 + \gamma_2)\mathds{1}_{\mathcal{O}_{1,d}\cap\mathcal{O}_{2,d}}
        \mathds{1}_{\omega_0}
    = 
        (\gamma_1\mathds{1}_{\mathcal{O}_{1,d}} + \gamma_2\mathds{1}_{\mathcal{O}_{2,d}})
        \mathds{1}_{\omega_0}
    =
    (-\psi_{t}-\psi_{xx}+a(x,t)\psi)
    \mathds{1}_{\omega_0}.
$$
Moreover, since  $\omega_0\subset\subset \mathcal{O}$, we can consider a function $\theta\in C^{2}(\mathbb{R};[0,1])$ satisfying
\begin{equation}
    \left\{
        \begin{array}{lll}
            \theta(x)=0 &\text{if}& x\in \omega_0,\\
            \theta(x)=1 &\text{if}& x\in \mathbb{R}\setminus\mathcal{O},
        \end{array}
    \right.
\end{equation}
and apply Lemma \ref{lemma-carlema1} with $i=0$ to deduce that
\begin{equation*}
\begin{array}{lll}
C \lambda\displaystyle\iint_{\omega_{0}\times(0,T)}e^{-2s\sigma_{0}}|h|^{2}dxdt
    =
        C \lambda\displaystyle\iint_{\omega_{0}\times(0,T)}e^{-2s\sigma_{0}}\theta(\gamma_1+ \gamma_2)(-\psi_{t}-\psi_{xx}+a(x,t)\psi)dxdt
\\\noalign{\smallskip}\phantom{C \lambda\displaystyle\iint_{\omega_{0}\times(0,T)}e^{-2s\sigma_{0}}|h|^{2}dxdt}
    \leq
        \dfrac{1}{2}I^{0}_{0}(h)
        +
        C s^{4}\lambda^{5}\displaystyle\iint_{\mathcal{O}\times(0,T)}e^{-2s\sigma_{0}}(\xi_{0})^{4}|\psi|^{2}dxdt,
\end{array}
\end{equation*}
for $s$ and $\lambda$ large enough.
Substituting this last inequality into \eqref{T2.1G_3}, we have
\begin{equation}
\begin{array}{lll}
I_{3}^{0}(\psi) + I^{0}_{0}(h)
    \leq
        Cs^{4}\lambda^{5}\displaystyle\iint_{\mathcal{O}\times(0,T)}e^{-2s\sigma_{0}}(\xi_{0})^{4}|\psi|^{2}dxdt,
\end{array}
\end{equation}
which proves item \ref{(1)}.

\qed

\subsection{Proof of Theorem \ref{newcarleman}: case \ref{(2)}}

Let $\widetilde{\mathcal{O}}\subset\subset\mathcal{O}$, and let us consider the cut-off function $\theta^0 \in C^{2}_0(\mathbb{R};[0,1])$ such that
\begin{equation}
    \left\{
        \begin{array}{lll}
            \theta^{0}(x)=0 &\text{if}& x\in \tilde{\mathcal{O}},\\
            \theta^{0}(x)=1 &\text{if}& x\in \mathbb{R}\setminus\mathcal{O},
        \end{array}
    \right.
\end{equation}
and the solution $(\psi,\gamma_1,\gamma_2)$ of system \eqref{adjointoptimalsystem}.
From $\eqref{adjointoptimalsystem}$, we obtain $\theta^0\psi$ solves
\begin{equation}\label{T2.1G_4}
\left\{\begin{aligned}
&(\theta^{0}\psi)_{t}-(\theta^{0}\psi)_{xx}+a(x,t)(\theta^{0}\psi)=\sum_{i=1}^{2}\theta^{0}\gamma^{i}\mathds{1}_{\mathcal{O}_{i,d}}+\theta^{0}_{xx}\psi -2(\psi\theta^{0}_{x})_x &&\text{in}&& Q_{\ell},\\
&\theta^{0}\psi(0,t)=\theta^{0}\psi(\ell(t),t)=0 &&\text{in} && (0,T),\\
&\theta^{0}\psi(x,T)=\theta^{0}\psi^{T}(x) &&\text{in} && (0,\ell(T)).
		\end{aligned}
		\right.
\end{equation}

Since  $\mathcal{O}\cap\mathcal{O}_{i,d}\neq\emptyset$, we can take the sets $\omega_i\subset\subset\widetilde{\mathcal{O}}\cap\mathcal{O}_{i,d}$, $i=1,2$,  and consider their associated weights $\eta_1^*, \eta_2^*$ furnished by Lemma \ref{Furshikov-lemma}. Then, applying the Carleman estimate of Lemma \ref{standardcarleman_} with $m=3$ and $i=1$ to \eqref{T2.1G_4},
 we obtain
\begin{equation}\label{TC-30}
\begin{split}
\displaystyle I^1_3(\theta^0\psi) & \leq C\left(s^3\lambda^4 \iint_{\omega_1\times(0,T)} e^{-2s\sigma_1} (\xi_1)^3 |\theta^0\psi|^2\, dxdt + \iint_{Q_\ell} e^{-2s\sigma_1} |\theta^0_{xx}|^2|\psi|^2\, dxdt \right.\\ 
& \displaystyle \quad \left. + s^2\lambda^2 \iint_{Q_\ell} e^{-2s\sigma_1} (\xi_1)^2 |\theta^0_x\psi|^2\, dxdt + \mathcal{P} \right) \\
& \displaystyle\leq C \left(2s^3\lambda^4 \iint_{\mathcal{O}\times(0,T)} e^{-2s\sigma_1} (\xi_1)^3 |\psi|^2\, dxdt + \mathcal{P}  \right),
\end{split}
\end{equation}
for $s$ and $\lambda$ large enough, where (recalling that $\sigma_1=\sigma_2$ in $\overline{Q}_\ell\setminus(\tilde{\mathcal{O}}\times(0,T))$)
\begin{equation*}
    \begin{array}{lll}
        \mathcal{P}
            =
               \displaystyle \iint_{Q_\ell} e^{-2s\sigma_1}|\theta^0|^2\left|\sum_{j=1}^2 \gamma^j\mathds{1}_{\mathcal{O}_{j,d}}\right|^2\, dxdt
    %
    %
            =   
                \displaystyle \iint_{Q_\ell} e^{-2s\sigma_2} |\theta^0|^2\left|\sum_{j=1}^2 \gamma^j\mathds{1}_{\mathcal{O}_{j,d}}\right|^2\, dxdt.
    \end{array}
\end{equation*}

In this case, we assume $\mathcal{O}_{1,d}\cap\mathcal{O}\neq\mathcal{O}_{2,d}\cap\mathcal{O}$, which leads to three possible scenarios:
\begin{enumerate}[i.]
    \item $(\mathcal{O}_{1,d}\cap\mathcal{O})\cap(\mathcal{O}_{2,d}\cap\mathcal{O})\neq (\mathcal{O}_{1,d}\cap\mathcal{O})$ and $(\mathcal{O}_{1,d}\cap\mathcal{O})\cap(\mathcal{O}_{2,d}\cap\mathcal{O})\neq (\mathcal{O}_{2,d}\cap\mathcal{O})$,
    \item $(\mathcal{O}_{1,d}\cap\mathcal{O})\cap(\mathcal{O}_{2,d}\cap\mathcal{O})\neq (\mathcal{O}_{1,d}\cap\mathcal{O})$ and $(\mathcal{O}_{1,d}\cap\mathcal{O})\cap(\mathcal{O}_{2,d}\cap\mathcal{O}) = (\mathcal{O}_{2,d}\cap\mathcal{O})$,
    \item $(\mathcal{O}_{1,d}\cap\mathcal{O})\cap(\mathcal{O}_{2,d}\cap\mathcal{O})= (\mathcal{O}_{1,d}\cap\mathcal{O})$ and $(\mathcal{O}_{1,d}\cap\mathcal{O})\cap(\mathcal{O}_{2,d}\cap\mathcal{O})\neq (\mathcal{O}_{2,d}\cap\mathcal{O})$.
\end{enumerate}
Thus, we will choose the sets $\omega_i, \ i=1,2$, satisfying, besides $\omega_i\subset \mathcal{O}\cap\mathcal{O}_{i,d}$, the following properties according to each scenario:
\begin{equation}\label{TC-31}
    \begin{gathered}
        \text{i.}\ \omega_1\cap\mathcal{O}_{2,d}\neq \emptyset,\ \omega_2\cap\mathcal{O}_{1,d}=\emptyset, \qquad \text{ii.} \ \omega_1\cap\mathcal{O}_{2,d} = \emptyset, \omega_2 \subset \mathcal{O}_{1,d},\\
        \text{iii.}\ \omega_1\subset\mathcal{O}_{2,d}, \ \omega_2\cap\mathcal{O}_{1,d}=\emptyset.
    \end{gathered}
\end{equation}
In each scenario, we will consider the constants $m^1_1,\ m^1_2,\ m^2_1, \ m^2_2$ given by
\begin{equation}
    \begin{gathered}
\text{i.} \ m^1_1 = 1, \quad m^1_2 = 0, \quad m^2_1 = 0, \quad m^2_2 = 1,\\
\text{ii.} \ m^1_1 = 1, \quad m^1_2 = 0, \quad m^2_1 = 1, \quad m^2_2 = 1,  \\
\text{iii.} \ m^1_1 = 1, \quad m^1_2 = 1, \quad m^2_1 = 0, \quad m^2_2 = 1. 
    \end{gathered}
\end{equation}
Then, defining $h_1 = m^1_1\gamma^1+m^1_2\gamma^2$ and $h_2 = m^2_1\gamma^1+m^2_2\gamma^2$, we have from \eqref{TC-31} that in all three possible scenarios it holds that
\begin{equation}\label{TC-32_1}
    h_i = \sum_{j=1}^2 \gamma^j\mathds{1}_{\mathcal{O}_{j,d}}
\end{equation}
and
\begin{equation}\label{TC-32_2}
    \mathcal{P}
    \leq
    \displaystyle \iint_{Q_\ell} \left(e^{-2s\sigma_1}|h_1|^2+e^{-2s\sigma_2}|h_2|^2\right)\, dxdt
    \leq \sum_{i=1}^2 I^i_0(h_i),
\end{equation}
for $s$ and $\lambda$ large enough. 

Now, in order to complete the estimate of $\mathcal{P}$, we will obtain a suitable bound for $I^i_0(h_i)$.
From Lemma \ref{standardcarleman_}, the expression \eqref{TC-32_1}, and recalling that $\sigma_1=\sigma_2$ and $\xi_1=\xi_2$ in $Q_\ell\setminus(\tilde{\mathcal{O}}\times(0,T))$, we have that
\begin{equation}\label{T2.1G_5}
    \begin{array}{lll}
    I^i_0(h_i)
        \leq
            C\lambda \displaystyle\iint_{\omega_i\times (0,T)} e^{-2s\sigma_i}|h_i|^2\,dxdt 
            + Cs^{-3}\lambda^{-3}\displaystyle\iint_{Q_\ell\setminus(\mathcal{O}\times(0,T))} e^{-2s\sigma_i}(\xi_i)^{-3}|\theta^0\psi|^2\, dxdt
    \\\noalign{\smallskip}\phantom{I^i_0(h_i)\leq}
            + Cs^{-3}\lambda^{-3}\displaystyle\iint_{\mathcal{O}\times(0,T)} e^{-2s\sigma_i}(\xi_i)^{-3}|\psi|^2\, dxdt
    \\\noalign{\smallskip}\phantom{I^i_0(h_i)}
        \leq 
            C\lambda \displaystyle\iint_{\omega_i\times (0,T)} e^{-2s\sigma_i}(1-\theta^0) h_i(-\psi_t -\psi_{xx}+a\psi)\,dxdt 
            +
            \lambda^{-1}I_3^1(\theta^0\psi)
    \\\noalign{\smallskip}\phantom{I^i_0(h_i)\leq}
            +
            Cs^{-3}\lambda^{-3}\displaystyle\iint_{\mathcal{O}\times(0,T)} e^{-2s\sigma_i}(\xi_i)^{-3}|\psi|^2\, dxdt,
    \end{array}
\end{equation}
for $s$ and $\lambda$ large enough. Moreover, from Lemma \ref{lemma-carlema1} we get
\begin{equation}\label{T2.1G_6}
    \begin{array}{lll}
         C\lambda \displaystyle\iint_{\omega_i\times (0,T)} e^{-2s\sigma_i}(1-\theta^0) h_i(-\psi_t -\psi_{xx}+a\psi)\,dxdt 
    \\\noalign{\smallskip}\phantom{000000000000}
            \leq
                 \dfrac{1}{2}I_0^i(h_i) + 
                 Cs^{4}\lambda^{5}\displaystyle\iint_{\mathcal{O}\times(0,T)} e^{-2s\sigma_i}(\xi_i)^{4}|\psi|^2\, dxdt,
    \end{array}
\end{equation}
for $s$ and $\lambda$ large enough. Hence, combining  \eqref{TC-32_2}–\eqref{T2.1G_6}, we obtain
\begin{equation}\label{T2.1G_7}
    \begin{array}{lll}
         \mathcal{P}
            \leq 
                2\lambda^{-1}L_3^1(\theta^0\psi)
                + 
                Cs^{4}\lambda^{5}\displaystyle\sum_{i=1}^2\iint_{\mathcal{O}\times(0,T)} e^{-2s\sigma_i}(\xi_i)^{4}|\psi|^2\, dxdt,
    \end{array}
\end{equation}
for $s$ and $\lambda$ large enough.

Taking $\lambda$ large enough, we deduce the desired inequality of case \ref{(2)} using \eqref{TC-30} and \eqref{T2.1G_7}. This concludes the proof of Theorem \ref{newcarleman}.

\qed
\section{Observability inequality}\label{observability}

We prove in this section the observability inequality stated in Theorem \ref{obsth}. We begin with an auxiliary energy lemma.

\begin{lemma}\label{energy}
There exists $t_0\in(0,T)$, depending only on $\mu_i$ and  $\|a\|_\infty$,
such that the solution $(\psi,\gamma_1,\gamma_2)$ of \eqref{adjointoptimalsystem} satisfies
\begin{eqnarray}\label{energyestimatepsi}
\int_{0}^{\ell(t)}|\psi(\cdot,t)|^2\,dx\leq
2\displaystyle\int_{0}^{\ell(t_{1})}|\psi(\cdot,t_{1})|^2\,dx,
\quad \forall\ 0\leq t < t_{1} \leq t_0.
\end{eqnarray}
In particular
\begin{eqnarray}\label{energyestimatepsi2}
\int_{0}^{\tfrac{t_0}{2}}\int_{0}^{\ell(t)}|\psi|^2\,dxdt
\leq
\dfrac{2t_0}{T-t_0}\int_{\tfrac{t_0}{2}}^{\tfrac{T}{2}}\int_{0}^{\ell(t)}|\psi|^2\,dxdt.
\end{eqnarray}
\end{lemma}
\begin{proof}
Multiplying $\eqref{adjointoptimalsystem}_1$ by $\psi$ and integrating over $(0,\ell(s))$, we have
\begin{equation}\label{le201_}
\begin{array}{lll}
-\displaystyle\dfrac{d}{ds}\int_0^{\ell(s)}|\psi(\cdot,s)|^2 dx
+\ell'(s)|\psi(\ell(s),s)|^2
\\\noalign{\smallskip}\phantom{0000}
=
2\displaystyle\int_{0}^{\ell(s)}\left(
-|\psi_{x}(\cdot,s)|^2
-  a|\psi(\cdot,s)|^2
+\sum_{i=1}^2 (\gamma^{i}\mathds{1}_{\mathcal{O}_{i,d}}\psi)(\cdot,s) \right)dx
\\\noalign{\smallskip}\phantom{0000}
\leq
\left(2\|a\|_\infty+2\right)
\displaystyle\int_0^{\ell(s)}|\psi(\cdot,s)|^2dx
+\displaystyle\int_0^{\ell(s)}\sum_{i=1}^2|\gamma^{i}(\cdot,s)|^2dx
, \quad \forall s\in(0,T).
\end{array}
\end{equation}
Multiplying $\eqref{adjointoptimalsystem}_2$ by $\gamma_i$ and integrating over $(0,\ell(s))$, we have
\begin{equation}\label{le202_}
\begin{array}{lll}
\displaystyle\dfrac{d}{ds}\int_0^{\ell(s)}|\gamma_i(\cdot,s)|^2 dx
-\ell'(s)|\gamma_i(\ell(s),s)|^2
\\\noalign{\smallskip}\phantom{00}
=
2\displaystyle\int_{0}^{\ell(s)}\left(
-|(\gamma_i)_{x}(\cdot,s)|^2
-  a|\gamma_i(\cdot,s)|^2
-\dfrac{1}{\mu_i}(\psi\mathds{1}_{\mathcal{O}_i}\gamma_i)(\cdot,s) \right)dx
\\\noalign{\smallskip}\phantom{00}
\leq
\left(2\|a\|_\infty+\mu_i^{-1}\right)
\displaystyle\int_0^{\ell(s)}|\gamma_i(\cdot,s)|^2dx
+\int_0^{\ell(s)}|\psi(\cdot,s)|^2dx
, \quad \forall s\in(0,T).
\end{array}
\end{equation}
Observing that $\psi(\ell(s),s)=\gamma_i(\ell(s),s)=\gamma_i(\cdot,0)=0$, we take $t\in(0,t_1)$ and integrate \eqref{le201_} over $(t,t_1)$ and \eqref{le202_} in $(0,t)$, and then add the resulting expressions to obtain
\begin{equation}\label{le203_}
\begin{array}{lll}
\displaystyle\int_0^{\ell(t)}|\psi(\cdot,t)|^2dx
+\sum_{i=1}^2\int_0^{\ell(t)}|\gamma_i(\cdot,t)|^2dx
\\\noalign{\smallskip}\phantom{0000}
\leq
\displaystyle\int_0^{\ell(t_1)}|\psi(\cdot,t_1)|^2dx
+\delta
\int_0^{t_1}\int_0^{\ell(s)}
\left(
|\psi(\cdot,s)|^2\,dx +\displaystyle\sum_{i=1}^2|\gamma^{i}(\cdot,s)|^2
\right)dxds,
\end{array}
\end{equation}
where $\delta=2\|a\|_\infty+4+\mu_1^{-1}+\mu_2^{-2}$.

Finally, integrating \eqref{le203_} with respect to the variable $t$ over $(0,t_1)$, we obtain
\begin{equation}\label{le204_}
\begin{array}{lll}
\dfrac{1-\delta t_1}{t_1}
\displaystyle\int_0^{t_1}\int_0^{\ell(t)}
\left(
|\psi(\cdot,t)|^2\,dx +\displaystyle\sum_{i=1}^2|\gamma^{i}(\cdot,t)|^2
\right)dxdt
\leq
\displaystyle\int_0^{\ell(t_1)}|\psi(\cdot,t_1)|^2dx.
\end{array}
\end{equation}
Thus, \eqref{energyestimatepsi} follows by comparing \eqref{le203_} and \eqref{le204_} if
$$
0\leq\delta\leq\dfrac{1-\delta t_1}{t_1},
$$
which holds for all $t_1\leq t_0: =\tfrac{1}{2\delta}$.
\end{proof}
\subsection{Proof of Theorem \ref{obsth}} 

Let $t_0\in(0,T)$ be the constant furnished by Lemma \ref{energy}. Taking $t=0$ in \eqref{energyestimatepsi} and integration in $t_{1}$ from $t_{0}/2$ to $t_{0}$, we obtain 
\begin{equation}\label{ob1_}
\int_{0}^{\ell_{0}}|\psi(x,0)|^{2}dx\leq C\iint_{Q_{t_0}}|\psi(x,t)|^{2}dxdt,
\end{equation}
where
$$Q_{t_0}:=\left\{(x,t)\in\mathbb{R}^{2}: 0\leq x\leq \ell(t), \tfrac{t_{0}}{2}\leq t\leq t_0\right\}.$$ 
Moreover, since there exists  $\delta_{t_0}>0$ such that
$e^{-2s\sigma_{j}}(\xi_{j})^{3}>\delta_{t_0}$ in $Q_{t_0}$ for $j=0,1,$
\begin{equation}\label{ob2_}
\iint_{Q_{t_0}}|\psi(x,t)|^{2}dxdt
\leq
\dfrac{1}{\delta_{t_0}}\iint_{Q_\ell}e^{-2s\sigma_{j}}(\xi_{j})^{3}|\psi(x,t)|^{2}dxdt,
\quad j=0,1.
\end{equation}
Hence, from \eqref{ob1_} and \eqref{ob2_}, we obtain
\begin{equation}\label{ob_1+2_}
\int_{0}^{\ell_{0}}|\psi(x,0)|^{2}dx\
\leq
C\iint_{Q_\ell}e^{-2s\sigma_{j}}(\xi_{j})^{3}|\psi(x,t)|^{2}dxdt,
\quad j=0,1.
\end{equation}
Next, fix $\varepsilon\in (0,T/2)$ and consider the function
\begin{equation}\label{rhosmall}
\rho_j(t)=\left\{
    \begin{array}{cl}\displaystyle
1 \ \ &\text{for} \ \ t\in\left(0,T-\varepsilon\right),
\\\noalign{\smallskip}
\displaystyle\max_{x\in[0,B]} e^{s\sigma_j(x,t)} \ \ &\text{for} \ \ t\in\left(T-\varepsilon,T\right)\subset(T/2,T),
    \end{array}
    \right.
\end{equation}
for $j=0,1$.
We multiply $\eqref{adjointoptimalsystem}_{2}$ by $\rho_j^{-2}\gamma^{i}$ and perform integration by parts to obtain that
\begin{equation}\label{ob3_}
    \begin{array}{lll}
        \displaystyle\int_{0}^{\tilde{t}}\int_{0}^{\ell(t)}\gamma_{t}^{i}\rho_j^{-2}\gamma^{i}dxdt
            =
        \displaystyle\int_{0}^{\tilde{t}}\int_{0}^{\ell(t)}\rho_j^{-2}
        \left(
            -|\gamma^{i}_{x}|^{2}-a(x,t)|\gamma^{i}|^{2}-\dfrac{1}{\mu_{i}}\gamma^{i}\psi\mathds{1}_{\mathcal{O}_{i}}
        \right)dxdt
\\\noalign{\smallskip}\phantom{\displaystyle\int_{0}^{\tilde{t}}\int_{0}^{l(t)}\gamma_{t}^{i}\rho_j^{-2}\gamma^{i}dxdt}
\leq
C\displaystyle\int_{0}^{\tilde{t}}\int_{0}^{\ell(t)}\rho_j^{-2}|\gamma^{i}|^{2}dxdt
+\iint_{Q_\ell}\rho_j^{-2}|\psi|^{2}dxdt.
\end{array}
\end{equation}
Since $\rho_j$ is non-decreasing in $(T/2,T)$ (constant on $(T/2-T-\varepsilon)$) and increasing on $(T-\varepsilon,T)$, we have $(\rho_j)_t\geq 0$ in $Q_\ell$. Using this fact, that $\gamma^{i}(\ell(\cdot),\cdot)=0$, and that $\rho_j>0$, we compute 
\begin{eqnarray}\label{ob4_}
\int_{0}^{\tilde{t}}\int_{0}^{\ell(t)}\gamma^{i}_{t}\gamma^{i}\rho^{-2}_{j}dxdt&=&\int_{0}^{\tilde{t}}\left(\dfrac{1}{2}\dfrac{d}{dt}\int_{0}^{\ell(t)}\rho_j^{-2}(t)|\gamma^{i}(x,t)|^2dx-\dfrac{1}{2}\ell'(t)\rho_j^{-2}(t)|\gamma^{i}(\ell(t),t)|^{2}\right.\nonumber\\
    &\quad&\left.\dfrac{1}{2}\int_{0}^{\ell(t)}\rho_j^{-3}(t)\rho_{t}(t)|\gamma^{i}(x,t)|^{2}dx  \right)dt\nonumber\\
    &\geq&\dfrac{1}{2}\int_{0}^{\ell(\tilde{t})}\rho_j^{-2}(\tilde{t})|\gamma^{i}(x,\tilde{t})|^{2}dx.
\end{eqnarray}
Moreover, in order to estimate the integral with $\psi$ in \eqref{ob3_}, we use the inequality \eqref{energyestimatepsi2} of Lemma \ref{energy}, that
$e^{-2s\sigma_j}>\delta_{t_0,\varepsilon}>0$ in $$Q_{t_0,\varepsilon}:=\left\{(x,t)\in\mathbb{R}^{2}: 0\leq x\leq \ell(t), \tfrac{t_{0}}{2}\leq t\leq T-\varepsilon\right\},\qquad j=0,1,$$ 
and the definition of $\rho_j$ to obtain 
\begin{eqnarray}\label{ob5_}
\iint_{Q_\ell}\rho_j^{-2}|\psi|^{2}dxdt&=&\int_{0}^{\tfrac{t_0}{2}}\int_{0}^{\ell(t)}|\psi|^{2}dxdt+\iint_{ Q_{t_0,\varepsilon}}|\psi|^{2}dxdt+\int_{\tfrac{T}{2}}^{T}\int_{0}^{\ell(t)}\rho_j^{-2}|\psi|^{2}dxdt\nonumber\\
&\leq& C \iint_{ Q_{t_0,\varepsilon}}|\psi|^{2}dxdt+\int_{T-\varepsilon}^{T}\int_{0}^{\ell(t)}e^{-2s\sigma_j}|\psi|^{2}dxdt\nonumber\\ 
&\leq&C(\varepsilon)\displaystyle\iint_{ Q_{\ell}}e^{-2s\sigma_j}|\psi|^{2}dxdt 
\end{eqnarray}
Finally, we combine \eqref{ob3_}–\eqref{ob5_} to obtain that
\begin{equation}\label{ob6_}
\begin{array}{lll}
\displaystyle\int_{0}^{\ell(\tilde{t})}\rho_j^{-2}(\tilde{t})|\gamma^{i}(x,\tilde{t})|^{2}dx
\\\noalign{\smallskip}\phantom{0000}
\leq
C\displaystyle\int_{0}^{\tilde{t}}\int_{0}^{\ell(t)}\rho_j^{-2}|\gamma^{i}|^{2}dxdt
+
C\displaystyle\iint_{Q_\ell}e^{-2s\sigma_j(x,t)}|\psi|^{2}dxdt,
\quad \forall \tilde{t}\in(0,T).
\end{array}
\end{equation}
Then, applying  Gronwall's inequality to \eqref{ob6_} and adding the result to \eqref{ob_1+2_}, we obtain
\begin{equation}\label{ob7_}
\begin{array}{lll}
\displaystyle
\int_{0}^{\ell_{0}}|\psi(x,0)|^{2}dx
+
\int_{0}^{\ell(\tilde{t})}\rho_j^{-2}(\tilde{t})|\gamma^{i}(x,\tilde{t})|^{2}dx
\leq
C\displaystyle\iint_{Q_\ell}e^{-2s\sigma_j(x,t)}|\psi|^{2}dxdt,
\quad \forall \tilde{t}\in(0,T).
\end{array}
\end{equation}
Hence, for any $\mu_1,\mu_2>\mu_{00}$, where $\mu_{00}$ is given in Theorem \ref{newcarleman}, we use the Carleman estimates of Theorem \ref{newcarleman}, with $j=0$ in case $\eqref{GC_2}_1$ and $j=1$ in case $\eqref{GC_2}_2$, to estimate the right-hand side of \eqref{ob7_} and reach the observability inequality \eqref{obsinequalityintro}.  This concludes the proof of the theorem.

\qed

\section{Approximate controllability for the linearized system}\label{approximatesec}

In this section, we prove Theorem \ref{aproxtheorem}. The proof is standard.

For each $\ell\in \mathcal{F}(\ell^*,\ell_0, B)$ and $\varepsilon\in(0,1)$, consider the continuous and strictly convex functional $F_{\varepsilon,\ell}: L^2(0,\ell(T)) \longrightarrow \mathbb{R}$ defined by
\begin{equation}\label{defFe}
    \begin{array}{lll}
    F_{\varepsilon,\ell}(\psi^T)
        =
            \displaystyle\int_{0}^{\ell_{0}} y_{0}(x)\,\psi(x,0)\,dx
            +
            \displaystyle\sum_{i=1}^{2}\iint_{\mathcal{O}_{i,d}\times(0,T)} y_{i,d}\gamma^{i}\,dx\,dt
    \\\noalign{\smallskip}\phantom{F_\varepsilon(\psi^T)=}
            +
            \dfrac{1}{2}\displaystyle\iint_{\mathcal{O}\times(0,T)} |\psi|^2 \,dx\,dt
            + \varepsilon\|\psi^T\|_{L^2(0,\ell(T))},
    \end{array}
\end{equation}
where $(\psi,\gamma_1,\gamma_2)$ is the solution of \eqref{adjointoptimalsystem} corresponding to with $\psi^T$. Since the assumptions of Theorem \ref{aproxtheorem} are consistent with those of Theorem \ref{obsth}, the latter result is applicable. Given the function $\rho$ of Theorem \ref{obsth} and the constant $C_0$ of the observability inequality \eqref{obsinequalityintro}, we have, using this observability inequality, that
\begin{equation}\label{Fe>-}
    \begin{array}{lll}
    F_{\varepsilon,\ell}(\psi^T)
        \geq
            - \dfrac{C_0}{2}\|y_0\|_{L^2(0,\ell_0)}^2
            - \dfrac{1}{2C_0}\displaystyle\int_{0}^{\ell_{0}} |\psi(x,0)|^2dx
            - \dfrac{C_0}{2}\displaystyle\sum_{i=1}^{2}\iint_{\mathcal{O}_{i,d}\times(0,T)} \rho^2 |y_{i,d}|^2dxdt 
    \\\noalign{\smallskip}\phantom{F_\varepsilon(\psi^T)=}
            - \dfrac{1}{2C_0}\displaystyle\sum_{i=1}^{2}\iint_{\mathcal{O}_{i,d}\times(0,T)} \rho^{-2} |\gamma^{i}|^2dxdt
            +
            \displaystyle\iint_{\mathcal{O}\times(0,T)} |\psi|^2 \,dx\,dt
            +
            \varepsilon\|\psi^T\|_{L^2(0,\ell(T))}

    \\\noalign{\smallskip}\phantom{F_\varepsilon(\psi^T)}
         \geq
            -\mathcal{C}\left(
                \|y_0\|_{L^2(0,\ell_0)}^2
                + \displaystyle\sum_{i=1}^{2}\iint_{\mathcal{O}_{i,d}\times(0,T)} \rho^2 |y_{i,d}|^2dxdt 
            \right)
                +
                \varepsilon\|\psi^T\|_{H_0^1(0,\ell(T))}
   \\\noalign{\smallskip}\phantom{F_\varepsilon(\psi^T)=}
                +\dfrac{1}{2}\displaystyle\iint_{\mathcal{O}\times(0,T)} |\psi|^2 \,dx\,dt,
    \end{array}
\end{equation}
where the constant $\mathcal{C}>0$ is independent of $\varepsilon\in(0,1)$ and $\ell\in\mathcal{F}(\ell^*,\ell_0, B)$. This proves that $F_{\varepsilon,\ell}$ is coercive. Therefore, $F_{\varepsilon,\ell}$ possesses a unique global minimizer $\psi^T_\varepsilon\in L^2(0,\ell(T))$.

If $\psi_\varepsilon^T\neq0$, we have from \eqref{connection} that 
\begin{equation}\label{F'=0}
    \begin{array}{lll}
        0
        =
        F_{\varepsilon,\ell}'(\psi^T_\varepsilon)\cdot\psi^T
    \\\noalign{\smallskip}\phantom{0}
        =
        \displaystyle\int_{0}^{\ell_{0}} y_{0}(x)\,\psi_\varepsilon(x,0)dx
        +
        \displaystyle\sum_{i=1}^{2}\iint_{\mathcal{O}_{i,d}\times(0,T)} y_{i,d}\gamma^{i}_\varepsilon dxdt
    \\\noalign{\smallskip}\phantom{0=}
        +
        \displaystyle\iint_{\mathcal{O}\times(0,T)} \psi_\varepsilon\psi dxdt
        +
        \displaystyle\dfrac{\varepsilon}{\|\psi^T_
        \varepsilon\|_{L^2(0,\ell(T))}}\int_0^{\ell(T)} \psi^T_\varepsilon\psi^T dxdt
    \\\noalign{\smallskip}\phantom{0}
        =
       \displaystyle\int_{0}^{\ell(T)} y(x,T)\psi^T(x)dx
        +
        \displaystyle\iint_{\mathcal{O}\times(0,T)} (\psi_\varepsilon-f)\psi dxdt
    \\\noalign{\smallskip}\phantom{0=}
        +
        \displaystyle\dfrac{\varepsilon}{\|\psi^T_
        \varepsilon\|_{L^2(0,\ell(T))}}\int_0^{\ell(T)} \psi^T_\varepsilon\psi^T dxdt,
\end{array}
\end{equation}
for all $\psi^T\in L^2(0,\ell(T))$, where $(\psi_\varepsilon,\gamma^1_\varepsilon,\gamma^2_\varepsilon)$ is the solution of \eqref{adjointoptimalsystem} associated with $\psi^T_\varepsilon$. 
Choosing the control $f_{\varepsilon,\ell}=\psi_\varepsilon \mathds{1}_\mathcal{O}$ in \eqref{F'=0}, we use energy estimates to obtain that the solution $(y_\varepsilon, \phi^1_\varepsilon, \phi^2_\varepsilon)$ of $\eqref{optimalsystem}_1$–$\eqref{optimalsystem}_5$ associated with $f_{\varepsilon,\ell}$ satisfies
\begin{equation}\label{F'=0_2}
    \begin{array}{lll}
       \displaystyle\int_{0}^{\ell(T)} y_\varepsilon(x,T)\psi^T(x)dx
       =
       \displaystyle\dfrac{\varepsilon}{\|\psi^T_
        \varepsilon\|_{L^2(0,\ell(T))}}\int_0^{\ell(T)} \psi^T_\varepsilon\psi^T dxdt
    \\\noalign{\smallskip}\phantom{\displaystyle\int_{0}^{\ell(T)}y_\varepsilon(x,T)\psi^T(x)dx}
    \leq
        \varepsilon \|\psi^T\|_{L^2(0,\ell(T))}\quad \forall \psi^T\in L^2(0,\ell(T)),
\end{array}
\end{equation}
from which \eqref{ye} follows.
 Moreover, using \eqref{Fe>-} and the fact that  $F_{\varepsilon,\ell}(\psi^T_\varepsilon)\leq F_{\varepsilon,\ell}(0)=0$, we obtain
    \begin{eqnarray}
    \|f_{\varepsilon,\ell}\|_{L^2(\mathcal{O}\times(0,T))}\leq \mathcal{C}\left(
             \|y_0\|_{H_0^1(0,\ell_0)}^2
             + \displaystyle\sum_{i=1}^{2} \|\rho y_{i,d}\|_{L^2(\mathcal{O}_{i,d}\times(0,T))}^2
                 \right),
      \end{eqnarray}
where the constant $\mathcal{C}>0$ is independent of $\varepsilon$ and $\ell$. Thus, \eqref{fe} is proved.

If $\psi_\varepsilon^T=0$, then we take $f_{\varepsilon,\ell}=0$ and, as a consequence, \eqref{fe} will hold. To observe that  \eqref{ye} also follows in this case, we substitute \eqref{connection} into \eqref{defFe} to obtain 
$$
F_{\varepsilon,\ell}(\psi^T) = D(\psi^T) + N_\varepsilon(\psi^T),
$$
where $D$ is the continuously differentiable functional
$$
D(\psi^T) = \displaystyle\int_{0}^{\ell(T)} y(x,T)\psi^T(x)dx
        +
        \displaystyle\iint_{\mathcal{O}\times(0,T)} (\psi-f)\psi dxdt
$$
and $N_\varepsilon$ is the functional given by
$$
N_\varepsilon(\psi^T) = \varepsilon\|\psi^T\|_{L^2(0,\ell(T))},
$$
which is not differentiable at $\psi^T=0$.
The minimum of $F_{\varepsilon,\ell}$ is then characterized by the variational inequality
\begin{equation}
    \begin{array}{lll}
          D'(\psi^T_\varepsilon)\cdot(\psi^T - \psi^T_\varepsilon)
          +
          N_\varepsilon(\psi^T) - N_\varepsilon(\psi^T_\varepsilon)\geq 0,
    \quad \forall \psi^T\in L^2(0,\ell(T)),
    \end{array}
\end{equation}
that is
\begin{equation}\label{vare}
    \begin{array}{lll}
        \displaystyle\int_{0}^{\ell(T)} y(x,T)(\psi^T-\psi^T_\varepsilon)(x)dx
        +
        \displaystyle\iint_{\mathcal{O}\times(0,T)} (\psi-f)(\psi- \psi_\varepsilon) dxdt
    \\\noalign{\smallskip}\phantom{0000000000000000}
        +
        \varepsilon\left(
            \|\psi^T\|_{L^2(0,\ell(T))}
            -
            \|\psi^T_\varepsilon\|_{L^2(0,\ell(T))}
        \right)
    \geq
        0,
    \quad \forall \psi^T\in L^2(0,\ell(T)).
    \end{array}
\end{equation}
Taking $\psi_\varepsilon^T=0$ and  $f_{\varepsilon,\ell}=0$ in \eqref{vare}, we obtain \eqref{ye}.

This concludes the proof of the theorem.

\qed

\section{Solving the multi-objective Stefan problem}\label{proofmainresult}

In this section, we prove Theorem~\ref{MT2}. The proof is based on a fixed-point argument that combines the approximate controllability of the linearized system $\eqref{optimalsystem}_1$–$\eqref{optimalsystem}_5$, given by Theorem~\ref{aproxtheorem}, with Schauder’s fixed-point theorem.

The compactness required by Schauder’s theorem will follow from a suitable parabolic H\"older regularity for the states $(y_\varepsilon, \phi^1_\varepsilon, \phi^2_\varepsilon)$ in Theorem~\ref{aproxtheorem}. More precisely, let us consider the subdomain  
$$
R_{\ell} := Q_{\ell} \cap \{{(x,t): x > \ell^{*}}\},
$$
of $Q_{\ell}$. Since $\mathcal{O},\mathcal{O}_1,\mathcal{O}_2\subset (0,\ell^*)$, the set $R_{\ell}$ does not intersect any of these. The parabolic  H\"older regularity result we use is stated next. 

\begin{lemma}\label{lemmmaregularity1}
Let $y_{0}\in W^{1,4}_{0}(0,\ell_{0})$, and let $(y_{\varepsilon,\ell}, \phi^1_{\varepsilon,\ell}, \phi^2_{\varepsilon,\ell})$ be a solution of $\eqref{optimalsystem}_1$–$\eqref{optimalsystem}_5$ furnished by Theorem \ref{aproxtheorem}. Then, there exist constants $\widetilde{C}_{1}>0$ and $\alpha=1/4$ such that
\begin{equation*}
(y_{\varepsilon,\ell}, \phi^1_{\varepsilon,\ell}, \phi^2_{\varepsilon,\ell}) \in [C^{1+\alpha,(1+\alpha)/2}(\overline{R_{\ell}})]^{3}
\end{equation*}
and
\begin{equation}
\begin{array}{lll}
\displaystyle\|y_{\varepsilon,\ell}\|_{C^{1+\alpha,\frac{1+\alpha}{2}}(\overline{R_{\ell}})}+\sum_{i=1}^{2}\|\phi^{i}_{\varepsilon,\ell}\|_{C^{1+\alpha,\frac{1+\alpha}{2}}(\overline{R_{\ell}})}
    \\\noalign{\smallskip}\phantom{000000000000}
    \leq
       \displaystyle \widetilde{C}_{1}\left(\|y_{0}\|_{W^{1,4}_{0}(0,\ell_{0})}+\sum_{i=1}^{2}\|\rho y_{i,d}\|_{L^{4}(\mathcal{O}_{i,d}\times(0,T))}\right).    
\end{array}
\end{equation}
In particular,
\begin{equation}\label{yxestimate_}
\begin{array}{lll}
\displaystyle\|(y_{\varepsilon,\ell})_x(\ell(\cdot),\cdot)\|_{C^{\alpha}([0,T])}
    \leq
       \displaystyle \widetilde{C}_{1}\left(\|y_{0}\|_{W^{1,4}_{0}(0,\ell_{0})}+\sum_{i=1}^{2}\|\rho y_{i,d}\|_{L^{4}(\mathcal{O}_{i,d}\times(0,T))}\right).    
\end{array}
\end{equation}

\end{lemma}
The proof of this lemma is similar to the proofs given in \cite{araujo2022remarks} and \cite{fernandez2016controllability}. For completeness, we include a brief proof in Appendix~\ref{regularityproperty}.

Since the assumptions of Theorem~\ref{MT2} are consistent with those of Theorem~\ref{aproxtheorem}, we may consider, for each $\varepsilon\in(0,1)$ and $\ell\in\mathcal{F}(\ell^{*},\ell_0,B)$, the control $f_{\varepsilon,\ell}$ and the associated unique solution $(y_{\varepsilon,\ell},\phi^1_{\varepsilon,\ell},\phi^2_{\varepsilon,\ell})$ of $\eqref{optimalsystem}_1$--$\eqref{optimalsystem}_5$ furnished by Theorem~\ref{aproxtheorem}. Then we have a well-defined, nonempty, bounded, closed, and convex subset of  $C^1([0,T])$
$$
\mathcal{F}_R = \left\{ \ell \in \mathcal{F}(\ell^*,\ell_0, B):\ \|\ell'\|_\infty \leq R \right\}, \quad R>0,
$$
and the operator $\Lambda_\varepsilon : \mathcal{F}_R \to C^1([0,T])$ given by
$$
\Lambda_\varepsilon(\ell)(t) = \ell_0 - \frac{1}{\beta}\int_0^t (y_{\varepsilon,\ell})_x(\ell(s),s) ds, \qquad t \in [0,T].
$$

We now aim to apply Schauder’s fixed-point theorem to the operator $\Lambda_\varepsilon$. To this end, we will show that, provided that $y_0, \rho, y_{i,d}$ are sufficiently small in norm, the operator $\Lambda_\varepsilon$ is  well defined, i.e., $\Lambda_\varepsilon(\mathcal{F}_R)\subset \mathcal{F}_R$, compact, and continuous.
\vspace{0.3cm}

\noindent\textbullet\quad $\Lambda_\varepsilon$ \textit{is well defined and compact.}

Applying the estimate \eqref{yxestimate_}, we obtain that
\begin{equation}\label{R_est}
    \begin{array}{lll}
        \|\Lambda_{\varepsilon}(\ell)\|_{C^{1+\alpha}([0,T])}
\\\noalign{\smallskip}\phantom{000000}
        =
        \|\Lambda_{\varepsilon}(\ell)\|_{C^0([0,T])}
        +
        \|\Lambda_{\varepsilon}'(\ell)\|_{C^{\alpha}([0,T])}
\\\noalign{\smallskip}\phantom{000000}
        \leq 
        \ell_0 + \dfrac{1}{\beta}(T+ 1)\|y_{\varepsilon,\ell}(\ell(\cdot),\cdot)\|_{C^{\alpha}([0,T])}
\\\noalign{\smallskip}\phantom{000000}
        \leq 
        \ell_0 + \dfrac{\widetilde{C}_{1}(T+ 1)}{\beta}\left(\|y_0\|_{W^{1,4}_{0}(0,\ell_{0})}+\displaystyle\sum_{i=1}^{2}\|\rho y_{i,d}\|_{L^{4}(\mathcal{O}_{i,d}\times(0,T))}\right),
%
        \end{array}
\end{equation}
and
\begin{equation}\label{lb_est}
    \begin{array}{lll}
        |\Lambda_{\varepsilon}(\ell)(t)-\ell_0|        
        \leq 
        \dfrac{1}{\beta}T\|y_{\varepsilon,\ell}(\ell(\cdot),\cdot)\|_{C^{\alpha}([0,T])}
\\\noalign{\smallskip}\phantom{|\Lambda_{\varepsilon}(\ell)(t)-\ell_0|}
        \leq 
        \dfrac{\widetilde{C}_{1}T}{\beta}\left(\|y_0\|_{W^{1,4}_{0}(0,\ell_{0})}+\displaystyle\sum_{i=1}^{2}\|\rho y_{i,d}\|_{L^{4}(\mathcal{O}_{i,d}\times(0,T))}\right),
%
    \end{array}
\end{equation}
for all $\ell\in\mathcal{F}(\ell^{*}, \ell_0, B)$. Then, taking the constant $\varepsilon_0$ in item $(2)$ of Theorem~\ref{MT2} as
\begin{equation}\label{epsilon0}
\varepsilon_{0}=\min\left\{\dfrac{\beta(R-\ell_0)}{\widetilde{C}_{1}(T+1)},\dfrac{\beta(B-\ell_0)}{\widetilde{C}_{1}T},\dfrac{\beta(\ell_0-\ell^*)}{\widetilde{C}_{1}T} \right\}>0,
\end{equation}
we have, from \eqref{R_est} and \eqref{lb_est}, that
\begin{equation}\label{est>T0}
    \begin{array}{c}
        \|\Lambda_{\varepsilon}(\ell)\|_{C^{1+\alpha}([0,T])}
        \leq
        R
       \qquad\mbox{and}\qquad
        \ell^*< \Lambda_{\varepsilon}(\ell)(t)< B,
        \quad \forall t\in [0,T],
    \end{array}
\end{equation}
for all $\ell\in\mathcal{F}(\ell^{*}, \ell_0, B)$.

Finally, from \eqref{est>T0}, we conclude that $\Lambda_\varepsilon(\mathcal{F}_R) \subset \mathcal{F}_R$ and that $\Lambda_\varepsilon(\mathcal{F}_R)$ is bounded in $C^{1+\alpha/2}([0,T])$. Therefore, the compactness of $\Lambda_{\varepsilon}(\mathcal{F}_R)$ in $C^{1}([0,T])$ is a direct consequence of such boundedness and the compact embedding $C^{1+\alpha/2}([0,T])\hookrightarrow C^{1}([0,T])$ .
\\

\noindent\textbullet\quad $\Lambda_\varepsilon$ \textit{ is continuous.}

In order to prove the continuity of $\Lambda_\varepsilon$, we take $\ell,\ell_n\in\mathcal{F}_R$, $n\geq 1$, such that $\|\ell_n-\ell\|_{C^1([0,T])}\rightarrow 0$, as $n\rightarrow \infty$. For the convergence $\Lambda_\varepsilon(\ell_n)\rightarrow \Lambda_\varepsilon(\ell)$, as $n\rightarrow\infty$, to hold, it suffices to prove that
\begin{equation}\label{yxconverge}
    (y_{\varepsilon,\ell_n})_x(\ell_n(\cdot),\cdot))
    \rightarrow
    (y_{\varepsilon,\ell})_x(\ell(\cdot),\cdot))
    \quad \mbox{in}\quad C^0([0,T]), \quad \mbox{as}\ n\rightarrow\infty.
\end{equation}
We divide the proof of \eqref{yxconverge} into two steps.

\subsection*{\textbf{Step 1. Convergence of controls}}
First, we will verify that the approximate controls $f_{\varepsilon,\ell_n}$, $f_{\varepsilon,\ell}$ given by Theorem \ref{aproxtheorem} satisfy 
\begin{equation}\label{TG0}
f_{\varepsilon,\ell_n}\rightarrow f_{\varepsilon,\ell} \quad\mbox{in}\quad L^2(\mathcal{O}\times(0,T)), \quad\mbox{as}\quad n\rightarrow\infty.
\end{equation}
For convenience, let $g^*$  denote the spatial extension of  $g$ by $0$ to the interval $(0,B)$.

Let us recall, from the proof of Theorem \ref{aproxtheorem}, that $(f_{\varepsilon,\ell_n},f_{\varepsilon,\ell})= (\psi_{\varepsilon,\ell_n}\mathds{1}_{\mathcal{O}},\psi_{\varepsilon,\ell}\mathds{1}_{\mathcal{O}})$, 
where $(\psi_{\varepsilon,\ell_n},\gamma^1_{\varepsilon,\ell_n},\gamma^2_{\varepsilon,\ell_n})$ (resp. $(\psi_{\varepsilon,\ell},\gamma^1_{\varepsilon,\ell},\gamma^2_{\varepsilon,\ell})$) is the solution of the adjoint system \eqref{adjointoptimalsystem}  corresponding to the unique global minimizer $\psi^T_{\varepsilon,\ell_n}\in L^2(0,\ell_n(T))$ (resp. $\psi^T_{\varepsilon,\ell}\in L^2(0,\ell(T))$) of  the functional $F_{\varepsilon,\ell_n}$ (resp. $F_{\varepsilon,\ell}$). Additionally, using \eqref{Fe>-} and that $F_{\varepsilon,\ell_n}(\psi^T_{\varepsilon,\ell_n})\leq F_{\varepsilon,\ell_n}(0)=0$, we obtain
\begin{equation}
        \|\psi^{T,*}_{\varepsilon,\ell_n}\|_{L^2(0,B)}
    =
        \|\psi^T_{\varepsilon,\ell_n}\|_{L^2(0,\ell_n(T))}
    \leq
        \dfrac{\mathcal{C}}{\varepsilon}\left(
            \|y_0\|_{H_0^1(0,\ell_0)}^2
            +
            \displaystyle\sum_{i=1}^{2} \|\rho y_{i,d}\|_{L^2(\mathcal{O}_{i,d}\times(0,T))}^2
        \right).
\end{equation}
Therefore, there exists $\hat{\psi}^T\in L^2(0,B)$ such that (up to a subsequence)
\begin{equation}\label{TG00}
    \psi^{T,*}_{\varepsilon,\ell_n}\rightharpoonup\hat{\psi}^T\quad \mbox{in}\quad L^2(0,B),
        \quad\mbox{as}\quad n\rightarrow\infty.
\end{equation}
%
Let us consider then
\begin{itemize}
    \item $(\psi_{\varepsilon,\ell_n},\gamma^1_{\varepsilon,\ell_n},\gamma^2_{\varepsilon,\ell_n})$, the solution of \eqref{adjointoptimalsystem} in $Q_{\ell_n}$ associated with $\psi^T_{\varepsilon,\ell_n}\in L^2(0,\ell_n(T))$;
    
    \item $(\psi_{\varepsilon,\ell},\gamma^1_{\varepsilon,\ell},\gamma^2_{\varepsilon,\ell})$, the solution of \eqref{adjointoptimalsystem} in $Q_{\ell}$ associated with $\psi^T_{\varepsilon,\ell}\in L^2(0,\ell(T))$;
    
    \item $(\hat{\psi},\hat{\gamma}^1,\hat{\gamma}^2)$, the solution of \eqref{adjointoptimalsystem} in $Q$ associated with $\hat{\psi}^T\in L^2(0,B)$;
    
    \item $(\bar{\psi},\bar{\gamma}^1,\bar{\gamma}^2)$, the solution of \eqref{adjointoptimalsystem} in $Q_\ell$ associated with $\bar{\psi}^T=\hat{\psi}^T|_{(0,\ell(T))}\in L^2(0,\ell(T))$;
    
    \end{itemize}
where $Q=(0,T)\times(0,B)$.
By observing the definition of a weak solution of \eqref{adjointoptimalsystem}, we can easily infer that
\begin{itemize}
    \item $(\psi_{\varepsilon,\ell_n}^*,\gamma^{1,*}_{\varepsilon,\ell_n},\gamma^{2,*}_{\varepsilon,\ell_n})$ is the solution of \eqref{adjointoptimalsystem} in $Q$ associated with $\psi^{T,*}_{\varepsilon,\ell_n} \in L^2(0,B)$;
    \item $(\bar{\psi}^*,\bar{\gamma}^{1,*},\bar{\gamma}^{2,*})$ is the solution of \eqref{adjointoptimalsystem} in $Q_\ell$ associated with $\bar{\psi}^{T,*}\in L^2(0,B)$.
\end{itemize}
Hence, it follows from \eqref{TG00} that
\begin{equation}\label{TG003}
    \begin{array}{lll}
    (\psi_{\varepsilon,\ell_n}^*,\gamma^{1,*}_{\varepsilon,\ell_n},\gamma^{2,*}_{\varepsilon,\ell_n})
            \rightarrow
        (\hat{\psi},\hat{\gamma}^1,\hat{\gamma}^2)
            \quad\mbox{in}\quad C^0([0,T];L^2(0,B)).
    \end{array}
\end{equation}
From the convergence in \eqref{TG003}, we have that
\begin{equation}\label{TG004}
    M_n\rightarrow M,\quad \mbox{as}\quad n\rightarrow\infty,
\end{equation}
where
\begin{equation*}
    \begin{array}{lll}
     M_n:= \displaystyle\int_{0}^{\ell_{0}} y_{0}(x)\,\psi_{n,\ell}^*(x,0)\,dx
            +
            \displaystyle\sum_{i=1}^{2}\iint_{\mathcal{O}_{i,d}\times(0,T)} y_{i,d}\gamma_{n,\ell}^{i,*}\,dx\,dt
            +
            \dfrac{1}{2}\displaystyle\iint_{\mathcal{O}\times(0,T)} |\psi_{n,\ell}^*|^2 \,dx\,dt,
        \\\noalign{\smallskip}
        M:= 
        \displaystyle\int_{0}^{\ell_{0}} y_{0}(x)\,\hat{\psi}(x,0)\,dx
            +
            \displaystyle\sum_{i=1}^{2}\iint_{\mathcal{O}_{i,d}\times(0,T)} y_{i,d}\hat{\gamma}^{i}\,dx\,dt
            +
            \dfrac{1}{2}\displaystyle\iint_{\mathcal{O}\times(0,T)} |\hat{\psi}|^2 \,dx\,dt.
    \end{array}
\end{equation*}
Moreover, since $(\psi_{\varepsilon,\ell_n}^*,\gamma^{1,*}_{\varepsilon,\ell_n},\gamma^{2,*}_{\varepsilon,\ell_n})=(\psi_{\varepsilon,\ell_n},\gamma^1_{\varepsilon,\ell_n},\gamma^2_{\varepsilon,\ell_n})$ in the sets $(0,T)\times(0,\ell_0)$, $(0,T)\times\mathcal{O}$, and $(0,T)\times\mathcal{O}_{i,d}$, we have
\begin{equation}\label{TG005}
    \begin{array}{lll}
        F_{\varepsilon,\ell_n}(\psi_{\varepsilon,\ell_n}^T)
        -
        \|\psi_{\varepsilon,\ell_n}^T\|_{L^2(0,\ell_n(T))}
        = M_n.
    \end{array}
\end{equation}
Also, since $\hat{\psi}^T\mathds{1}_{(\ell_n(T),B)}\rightarrow \hat{\psi}^T\mathds{1}_{(\ell(T),B)}$ strongly in $L^2(0,B)$, we have 
$$
0
=
\displaystyle\lim_{n\rightarrow\infty} \int_0^B \psi^{T,*}_{\varepsilon,\ell_n}\hat{\psi}^T\mathds{1}_{(\ell_n(T),B)}
=
\int_0^B \hat{\psi}^T \hat{\psi}^T\mathds{1}_{(\ell(T),B)}
=
\|\hat{\psi}^T\|_{L^2(\ell(T),B)}^2,
$$
that is, $\hat{\psi}^T=\bar{\psi}^{T,*}$ in $L^2(0,B)$. Hence, by  the uniqueness of solution, we have that $(\hat{\psi},\hat{\gamma}^1,\hat{\gamma}^2)=(\bar{\psi}^*,\bar{\gamma}^{1,*},\bar{\gamma}^{2,*})$ in $C^0([0,T];L^2(0,B))$ and, as a consequence, that 
\begin{equation}
    \begin{array}{c}\label{TG005.1}
    \hat{\psi} = \bar{\psi}
    \ \mbox{in}\ 
    L^2(\mathcal{O}\times(0,T)),
    \\\noalign{\smallskip}
    \hat{\gamma}^i=\bar{\gamma}^i
    \ \mbox{in}\ 
    L^2(\mathcal{O}_{i,d}\times(0,T)),
        \quad\quad
    \hat{\psi}(\cdot,0) = \bar{\psi}(\cdot,0)
    \ \mbox{in}\ 
    L^2(0,\ell_0).
    \end{array}
\end{equation}
Thus,
\begin{equation}
    \begin{array}{lll}\label{TG006}
        F_{\varepsilon,\ell}(\bar{\psi}^T)
        -
        \|\bar{\psi}^T\|_{L^2(0,\ell(T))}
        = M.
    \end{array}
\end{equation}
Therefore, from \eqref{TG00}, \eqref{TG004}--\eqref{TG006}, we obtain
\begin{equation}\label{TG007}
F_{\varepsilon,\ell}(\bar{\psi}^T)
    \leq
        \liminf_{n\rightarrow\infty}  F_{\varepsilon,\ell_n}(\psi^T_{\varepsilon,\ell_n}).
\end{equation}

Now consider a sequence $\tilde{\psi}_n^T\in C_0^\infty(0,B)$ such that $\tilde{\psi}_n^T\rightarrow \phi^{T,*}_{\varepsilon,\ell}$ in $L^2(0,B)$. Then one verifies that the sequence $\psi_n^T=\tilde{\psi}_n^T\mathds{1}_{(0,\ell_n(T))}\in L^2(0,\ell_n(T))$ satisfies
\begin{equation}\label{TG008}
    \psi_n^{T,*}\rightarrow \phi^{T,*}_{\varepsilon,\ell}\ \mbox{in}\ L^2(0,B)
        \quad\mbox{as}\quad n\rightarrow\infty.
\end{equation}
Repeating the same arguments as those used after the convergence \eqref{TG00}, we obtain, from \eqref{TG008}, that
\begin{equation}\label{TG009}
    F_{\varepsilon,\ell}(\psi_{\varepsilon,\ell}^T)
    =
    \lim_{n\rightarrow\infty}  F_{\varepsilon,\ell_n}(\psi_n^T).
\end{equation}
Therefore, we get from \eqref{TG007}, \eqref{TG009}, and the minimality of $\psi^T_{\varepsilon,\ell_n}$, that 
\begin{equation}\label{TG0091}
    F_{\varepsilon,\ell}(\bar{\psi}^T)
    \leq
        \liminf_{n\rightarrow\infty}  F_{\varepsilon,\ell_n}(\psi^T_{\varepsilon,\ell_n})
    \leq
        \lim_{n\rightarrow\infty}  F_{\varepsilon,\ell_n}(\psi_n^T)
    =
    F_{\varepsilon,\ell}(\psi_{\varepsilon,\ell}^T).
\end{equation}
Hence, from the uniqueness of the minimizer $\psi^T_{\varepsilon,\ell}$, we obtain the states $\bar{\psi}^T$ and $\psi^T_{\varepsilon,\ell}$ are equal in $L^2(0,\ell(T))$. In particular, the solutions associated with such states satisfy
$\psi_{\varepsilon,\ell}=\bar{\psi}$ in $L^2(Q_\ell)$. Then, we deduce from \eqref{TG003} and $\eqref{TG005.1}_1$ that
$$
\psi_{\varepsilon,\ell_n}=\psi_{\varepsilon,\ell_n}^*\rightarrow \hat{\psi} = \bar{\psi} = \psi_{\varepsilon,\ell}\quad in\quad  \ L^2(\mathcal{O}\times(0,T)),\quad \mbox{as\ } n\rightarrow\infty,
$$
that is, $f_{\varepsilon,\ell_n}\rightarrow f_{\varepsilon,\ell}$ as $n\rightarrow\infty$.
\subsection*{\textbf{Step 2. Convergence of derivatives in the moving boundary}}
Our goal now is to prove the convergence \eqref{yxconverge}.
Let us denote
\begin{equation}
    \begin{array}{c}
         Q_{n,T_1,T_2}=\{(x,t)\in Q_{\ell_n};\ x\in(0,\ell_n(t)),\ t\in (T_1,T_2)\},
         \\\noalign{\smallskip}
         Q_{T_1,T_2}=\{(x,t)\in Q_{\ell};\ x\in(0,\ell(t)),\ t\in (T_1,T_2)\}.
    \end{array}
\end{equation}

Let us first analyze the case where $t\in[T_1,T_2]\subset[0,T]$ is such that $\ell(t)\leq\ell_n(t)$. The case $\ell(t)\geq \ell_n(t)$ can be handled in a completely analogous way, and will therefore be treated only briefly afterwards. In this case, we can evaluate $y_{\varepsilon,\ell_n}$ in $(\ell(t),t)$ and compute that
 \begin{equation}\label{TG010}
 \begin{array}{lll}
    \| (y_{\varepsilon,\ell_n})_x(\ell(\cdot),\cdot))
    -
    (y_{\varepsilon,\ell})_x(\ell(\cdot),\cdot)) \|_{C^0([T_1,T_2])}
\\\noalign{\smallskip}\phantom{00000000}
\leq
        \| (y_{\varepsilon,\ell_n})_x(\ell_n(\cdot),\cdot))
        -
        (y_{\varepsilon,\ell_n})_x(\ell(\cdot),\cdot)) \|_{C^0([T_1,T_2])}
        \\\noalign{\smallskip}\phantom{00000000=}
        + \| (y_{\varepsilon,\ell_n})_x(\ell(\cdot),\cdot))
        -
        (y_{\varepsilon,\ell})_x(\ell(\cdot),\cdot)) \|_{C^0([T_1,T_2])}
\\\noalign{\smallskip}\phantom{00000000}
\leq
    \|\ell_n-\ell\|_{C^0([0,T])}^\alpha
    \|(y_{\varepsilon,\ell_n})_x\|_{C^\alpha(\overline{R_{\ell}})}
    +
    \| (z_n)_x(\ell(\cdot),\cdot))\|_{C^0([T_1,T_2])},
 \end{array}
 \end{equation}
where $(z_n,p_n)=(y_{\varepsilon,\ell_n}-y_{\varepsilon,\ell},\ \phi^i_{\varepsilon,\ell_n}-\phi^i_{\varepsilon,\ell})$ solves
 \begin{equation*}
		\left\{
  \begin{array}{lll}
	(z_n)_t-(z_n)_{xx}+a(x,t)z_n = (f_{\varepsilon,\ell_n}- f_{\varepsilon,\ell})\mathds{1}_{\mathcal{O}}
    -\dfrac{1}{\mu_{1}}\phi^{1,n}\mathds{1}_{\mathcal{O}_{1}}
    -\dfrac{1}{\mu_{2}}\phi^{2,n}\mathds{1}_{\mathcal{O}_{2}}
    &\text{in}& Q_{T_1,T_2},
\\\noalign{\smallskip}
	-(p_n)_{t}-(p_n)_{xx}+a(x,t)p_n=z_n\mathds{1}_{\mathcal{O}_{i,d}} &\text{in}& Q_{T_1,T_2},
\\\noalign{\smallskip}
    z_n(0,t)=p_n(0,t)=0 &\text{in} & (T_1,T_2),
\\\noalign{\smallskip}
    z_n(\ell(t),t)=y_{\varepsilon,\ell_n}(\ell(t),t),\quad
    p_n(\ell(t),t)=\phi^{i}_{\varepsilon,\ell_n}(\ell(t),t) &\text{in} & (T_1,T_2),
\\\noalign{\smallskip}
	z_n(\cdot,T_1)=y_{\varepsilon,\ell_n}(\cdot,T_1)-y_{\varepsilon,\ell}(\cdot,T_1) &\text{in} & (0,\ell(T_1)),
\\\noalign{\smallskip}
	p_n(\cdot,T_2)=\phi^i_{\varepsilon,\ell_n}(\cdot,\ell(T_2))-\phi^i_{\varepsilon,\ell}(\cdot,\ell(T_2)) &\text{in} & (0,\ell(T_2)).
		\end{array}
		\right.
\end{equation*} 
Then, from Lemma \ref{lemmmaregularity1}, it follows that
\begin{equation}\label{TG011}
    \begin{array}{lll}
         \|(z_n)_x(\ell(\cdot),\cdot)\|_{C^0([T_1,T_2])}
    %
    \\\noalign{\smallskip}\phantom{0000}
            \leq
            C\left(\dfrac{}{}
                \|f_{\varepsilon,\ell_n}- f_{\varepsilon,\ell}\|_{L^2(\mathcal{O}\times(0,T))}
                +
                \|(y_{\varepsilon,\ell_n}-y_{\varepsilon,\ell})(\ell(\cdot),\cdot)\|_{L^2(T_1,T_2)}
    \right.
    \\\noalign{\smallskip}\phantom{0000=}
    \left.
                +
                \|(y_{\varepsilon,\ell_n}-y_{\varepsilon,\ell})(\cdot,T_1)\|_{L^2(0,\ell(T_1))}
                +
                \|(\phi^i_{\varepsilon,\ell_n}-\phi^i_{\varepsilon,\ell})(\cdot,\ell(T_2))\|_{L^2(0,\ell(T_2))}
            \dfrac{}{}\right)
    \\\noalign{\smallskip}\phantom{0000}
            \leq
            C\left(
                \|f_{\varepsilon,\ell_n}- f_{\varepsilon,\ell}\|_{L^2(\mathcal{O}\times(0,T))}
                +
                \|y_{\varepsilon,\ell_n}^*-y_{\varepsilon,\ell}^*\|_{C^0(\overline{Q})}
                +
                \|\phi^{i,*}_{\varepsilon,\ell_n}-\phi^{i,*}_{\varepsilon,\ell}\|_{C^0(\overline{Q})}
            \dfrac{}{}\right),
    \end{array}
\end{equation}
where, recalling, $Q=(0,T)\times(0,B)$. Since in Step $1$ we have proven that $f_{\varepsilon,\ell_n}\rightarrow f_{\varepsilon,\ell}$ in $L^2(\mathcal{O}\times(0,T))$, then, it is also true that
\begin{equation}\label{TG012}
\begin{array}{lll}
    (y_{\varepsilon,\ell_n}^*,\ \phi^{i,*}_{\varepsilon,\ell_n})
        \rightarrow
    (y_{\varepsilon,\ell}^*,\ \phi^{i,*}_{\varepsilon,\ell})
        \quad \mbox{in}\quad [C^0([0,T]; H_0^1(0,B))]^2\hookrightarrow[C^0(\overline{Q})]^2,
     \quad \mbox{as}\quad n\rightarrow\infty.
\end{array}
\end{equation}
Hence, using the convergences of $f_{\ell_n}$ and $\ell_n$, we get from \eqref{TG010}--\eqref{TG012} that
\begin{equation}
    (y_{\varepsilon,\ell_n})_x(\ell_n(\cdot),\cdot))
    \rightarrow
    (y_{\varepsilon,\ell})_x(\ell(\cdot),\cdot))
    \quad\mbox{in}\quad C([T_1,T_2]),\quad \mbox{as}\quad n\rightarrow\infty.
\end{equation}
In fact, the computations above ensure that the following convergence holds uniformly for all $t\in[0,T]$:
\begin{equation}\label{TG013}
    (y_{\varepsilon,\ell_n})_x(\ell_n(t),t))
    \rightarrow
    (y_{\varepsilon,\ell})_x(\ell(t),t)),
    \quad \mbox{as}\quad n\rightarrow\infty,
    \quad\mbox{if}\quad \ell(t)\leq\ell_n(t).
\end{equation}
This completes the analysis of the case $\ell_n(t)\leq\ell(t)$.

Now, if $t\in[T_1,T_2]\subset[0,T]$ is such that $\ell_n(t)\geq\ell(t)$, then we can consider $(y_{\varepsilon,\ell})_x(\ell_n(t),t))$ and proceed as before to deduce that 
\begin{equation}\label{TG014}
    (y_{\varepsilon,\ell_n})_x(\ell_n(t),t))
    \rightarrow
    (y_{\varepsilon,\ell})_x(\ell(t),t)),
    \quad \mbox{as}\quad n\rightarrow\infty,
    \quad\mbox{if}\quad \ell_n(t)\leq\ell(t).
\end{equation}
Therefore,  \eqref{yxconverge} follows from \eqref{TG013} and \eqref{TG014}.

This concludes the proof of the continuity of $\Lambda_\varepsilon$.
\\

We are now in a position to apply Schauder’s fixed-point theorem. Thus, there exists $\ell_\varepsilon\in \mathcal{F}_R$ such that $\ell_\varepsilon = \Lambda_\varepsilon(\ell_\varepsilon)$. In particular,
$$
\ell_\varepsilon'(t) = (y_{\varepsilon, \ell_\varepsilon})_x(\ell_\varepsilon(t),t),\quad \forall t\in[0,T].
$$
Therefore, the quintuple $(y_{\varepsilon, \ell_\varepsilon},\ell_\varepsilon,f_{\ell_\varepsilon},\phi^1_{\ell_\varepsilon},\phi^2_{\ell_\varepsilon})$ solves the Stefan problem \eqref{optimalsystem} and satisfies the approximate controllability condition
$$
    \|y_{\varepsilon,\ell_\varepsilon}(\cdot,T)\|\leq\varepsilon\quad \mbox{in}\quad  L^2(0,\ell_\varepsilon(T)).
$$

Finally, since $\mathcal{F}_R$ is compact, there exists $\hat{\ell}\in \mathcal{F}_R$ such that $\ell_\varepsilon\rightarrow \hat{\ell}$ in $C^1([0,T])$, as $\varepsilon\rightarrow0$. Arguing as in Step 1, we obtain that $f_{\ell_\varepsilon}\rightarrow f_{\hat{\ell}}$, as $\varepsilon\rightarrow0$. Arguing as in Step 2, we deduce that
\begin{equation}\label{TG015}
\begin{array}{lll}
    (y_{\varepsilon,\ell_\varepsilon}^*,\ \phi^{i,*}_{\varepsilon,\ell_\varepsilon})
        \rightarrow
    (y_{\hat{\ell}}^*,\ \phi^{i,*}_{\hat{\ell}})
        \quad \mbox{in}\quad [C^0([0,T]; H_0^1(0,B))]^2\hookrightarrow[C^0(\overline{Q})]^2,
     \quad \mbox{as}\quad \varepsilon\rightarrow 0,
\end{array}
\end{equation}
and
\begin{equation}\label{TG016}
    (y_{\varepsilon,\ell_\varepsilon})_x(\ell_\varepsilon(t),t))
    \rightarrow
    (y_{\hat{\ell}})_x(\hat{\ell}(t),t))
    \quad \mbox{uniformly,\quad  as}\quad \varepsilon\rightarrow 0.
\end{equation}
Therefore, the quintuple $(y_{\hat{\ell}},\hat{\ell},f_{\hat{\ell}},\phi^1_{\hat{\ell}},\phi^2_{\hat{\ell}})$ solves the Stefan problem \eqref{optimalsystem} and satisfies the null controllability condition
$$
    y_{\hat{\ell}}(\cdot,T)=0\quad\mbox{in}\quad L^2(0,\hat{\ell}(T)).
$$
This concludes the proof of the theorem.

\section{Comments and Conclusions}\label{conclusion}
\begin{itemize}

    \item It is possible to invert the leader-follower roles in the Stackelberg (or Pareto) strategy by taking $(v_{1},v_{2})$ as the leader and the null control $f$ as the follower. In this case, the problem would consist in finding a Nash equilibrium $(v_{1},v_{2})$ subject to the condition that $f$ is a null control, that is, finding $(v_{1},v_{2})$ that solves
\begin{equation*}
\left|
    \begin{array}{ll}\displaystyle
        J_1(f,v_{1},v_{2}) = \min_{\mathcal{N}_1(v_{2})} J_1(f,\hat v_{1}, v_{2}),
        \\ \noalign{\smallskip} 
        \displaystyle 
        J_2(f,v_{1},v_{2}) = \min_{\mathcal{N}_2(v_{1})} J_2(f,v_{1},\hat v_{2}),
    \end{array}
    \right.
\end{equation*}
where
\begin{equation*}
    \begin{array}{lll}
       \mathcal{N}^1(v_{2})=
    \left\{
    \bar{v}_1\in L^2(\mathcal{O}_1\times(0,T));\
    y(y_0,\bar{v}_1,v_{2},f,\ell;x,T)=0,\ f\in L^2(\mathcal{O}\times(0,T))
    \right\},
    \\\noalign{\smallskip}
     \mathcal{N}^2(v_{1})=
    \left\{
    \bar{v}_2\in L^2(\mathcal{O}_2\times(0,T));\
    y(y_0,v_{1},\bar{v}_2,f,\ell;x,T)=0,\ f\in L^2(\mathcal{O}\times(0,T))
    \right\},
    \end{array}
\end{equation*}
with  $y(y_0,\bar{v}_1,\bar{v}_2,f,\ell,x,t)$ denoting the solution of \eqref{mainequation}–\eqref{stefancontion} associated with $y_0,\bar{v}_1,\bar{v}_2,\ell$. A possible approach would be to follow the ideas of \cite{Luz1}. 

\item Our approach can be applied to other problems. For example, we may introduce a nonlinear term $F$ in the Stefan problem, following the ideas of \cite{fernandez2017local} and \cite{F-E-M}. Another interesting question concerns the viscous Burgers equation \cite{fernandez2017localB}, which can be addressed by combining the techniques in \cite{araruna2024bi}. Further questions arise from the fact that the moving boundary interacts with the heat flux through the relation $\ell'(t) = -\beta^{-1} y_{x}(\ell(t), t)$, which follows from the choice of the Laplacian operator $\mathcal{A}(y) = y_{xx}$. However, the relationship between $\ell$ and the operator $\mathcal{A}$ may change when other types of operators are considered. This leads to several interesting open questions regarding hierarchical control:
\begin{itemize}
    \item[i)] \textit{Quasi-linear 1D parabolic equation:} 
\begin{equation*}
	\left\{\begin{aligned}
		&y_t - (a(y)y_x)_{x}+F(y) =f\mathds{1}_{\mathcal{O}}+v_{1}\mathds{1}_{\mathcal{O}_{1}}+v_{2}\mathds{1}_{\mathcal{O}_{2}}  &&\text{in}&& Q_{\ell},\\ 
        &y(0,t)=y(\ell(t),t)=0      &&\text{in}&& (0,T),\\ 
		&y(\cdot,0) = y_0 &&\text{in}&& (0,\ell_{0}),
	\end{aligned}
	\right.
\end{equation*}
and consider the Stefan condition: 
\begin{equation*}
\ell'(t)=-a(y)y_{x}(\ell(t),t),
\end{equation*}
where $F$ is a given $C^{2}$ function defined on $\mathbb{R}$ with $F(0) = 0$ and $a(\cdot):\mathbb{R}\to \mathbb{R}$ is a twice continuously differentiable function.

    \item[ii)] \textit{Degenerate 1D parabolic equation:} Let us assume that $\gamma\in[0,2]$ is an exponent
\begin{equation*}
	\left\{\begin{aligned}
		&y_t - (x^{\gamma}y_x)_{x} =f\mathds{1}_{\mathcal{O}}+v_{1}\mathds{1}_{\mathcal{O}_{1}}+v_{2}\mathds{1}_{\mathcal{O}_{2}}  &&\text{in}&& Q_{\ell},\\ 
        &y(\ell(t),t)=0 \ \ \ \text{and}\ \ \ \left\{
\begin{aligned}
&y(0,t) = 0, && \text{(Weak)} \\
&\text{or}\\
&(x^{\gamma} y_x)(0,t) = 0, && \text{(Strong)}
\end{aligned}
\right.     &&\text{in}&& (0,T),\\ 
		&y(\cdot,0) = y_0 &&\text{in}&& (0,\ell_{0}),
	\end{aligned}
	\right.
\end{equation*}
and consider the Stefan condition: 
\begin{equation*}
\ell'(t)=-\ell(t)^{\gamma}y_{x}(\ell(t),t).
\end{equation*}
\end{itemize}
These problems remain open in our setting. It is important to highlight that, for degenerate systems, the case $(G2)$ remains open even on fixed domains. The analysis of these two cases requires a deeper understanding of how to treat the operator, the regularity properties of the Stefan condition, and the corresponding fixed-point arguments (see, for instance, \cite{ costa2023controllability, wang2022local}).

\end{itemize}

\appendix
\renewcommand{\thesection}{A}

\section*{Appendix}

This section contains auxiliary results, included for completeness or to provide necessary background. It contains results on the existence, uniqueness, and regularity of solutions to the linear optimality system \eqref{optimalsystem} and  the existence of the Fursikov weights employed in the Carleman inequality  (Lemma~\ref{Furshikov-lemma}).

\section{Proof of Lemma \texorpdfstring{\ref{Furshikov-lemma}}{2}}\label{AppendixA}

The proof of item \ref{F1} can be found in \cite{fernandez2016controllability}. To prove \ref{F2}, we denote $\omega_i=(a_i,b_i)$ and consider a nonempty interval $(a_i',b_i')\subset\subset (a_i,b_i)$. We observe that 
	$$
	0<\tilde{a}<a_i<a_i'<b_i'<b_i<\tilde{b}<\ell^{*}<\ell(t)<B.
	$$
	In this proof, we will build the functions $\eta_i^*$ step-by-step using cut-off functions $\theta\in C^\infty([0,B])$ with $0\leq \theta\leq 1$. To the best of our knowledge, this construction is new and is, moreover, very useful for applications such as numerical experiments.
	
First, we require that

\begin{equation}\label{lemmaeta*_1}
	\eta_i^*=\eta_2^* = \dfrac{x}{\tilde{a}} \mbox{\ in\ } [0,\tilde{a}],
	\quad\quad
	\eta_i^*=\eta_2^* = 1-\dfrac{x-\tilde{b}}{\ell(t)-\tilde{b}} \mbox{\ in\ } [\tilde{b},\ell(t)],\quad \forall t\in[0,T],
	\end{equation}
	so that $\ref{F2}_4$ is satisfied. To connect the two expressions in \eqref{lemmaeta*_1} in a $C^2$ way, we define a (preliminary) version of $\eta_i^*$ by
	$$
	\eta_i^*(x,t) = \left(\dfrac{x}{\tilde{a}}\right)\theta_{a,i}(x) + 	\left(1-\dfrac{x-\tilde{b}}{\ell(t)-\tilde{b}}\right)\theta_{b,i}(x),
	$$
	where $\theta_{a,i}, \theta_{b,i}$ are cut-off functions satisfying
	$$
	\theta_{a,i} = 1 \mbox{\ in\ } \left[0,a_i+\dfrac{1}{3}(a_i'-a_i)\right],
	\quad\quad
	\theta_{a,i} = 0 \mbox{\ in\ } \left[a_i'-\dfrac{1}{3}(a_i'-a_i),B\right],
	$$
	$$
	\theta_{b,i} = 0 \mbox{\ in\ } \left[0,b_i'+\dfrac{1}{3}(b_i-b_i')\right],
	\quad\quad
	\theta_{b,i} = 1 \mbox{\ in\ } \left[b_i-\dfrac{1}{3}(b_i-b_i'),B\right].
	$$
	Clearly, $\eta_i^*$ satisfies \ref{F2}, except for $\ref{F2}_1$. To address this issue, we include a cut-off function
	$$
	\theta_{c,i} = 0 \mbox{\ in\ } \left[0,a_i+\dfrac{1}{6}(a_i'-a_i)\right]\cup\left[b_i'-\dfrac{1}{6}(b_i-b_i'),B\right],
	$$
	$$
	\theta_{c,i} = 1 \mbox{\ in\ } \left[a_i+\dfrac{1}{6}(a_i'-a_i),b_i-\dfrac{1}{6}(b_i-b_i')\right].
	$$
	and redefine $\eta_i^*$ as 
	$$
	\eta_i^*(x,t) = \left(\dfrac{x}{\tilde{a}}\right)\theta_{a,i}(x) + 	\left(1-\dfrac{x-\tilde{b}}{\ell(t)-\tilde{b}}\right)\theta_{b,i}(x) + \theta_{c,i}(x).
	$$
	Now, $\eta_i^*\in C^\infty(\overline{Q_{\ell}})$ and fully satisfies \ref{F2}. It remains to ensure the equalities of norms $\|\eta^{*}_1\|_\infty = \|\eta^{*}_2\|_\infty$. For this purpose, we include another cut-off function
	$$
	\theta_{d,i} = 0 \mbox{\ in\ } \left[0,a_i'- \dfrac{1}{9}(a_i'-a_i)\right]\cup\left[b_i'+\dfrac{1}{9}(b_i-b_i'),B\right],
	\quad\quad
	\theta_{d,i} = 1 \mbox{\ in\ } \left[a_i',b_i'\right].
	$$
	and obtain the final version of $\eta_i^*$ as
	\begin{equation}
	\eta_i^*(x,t) = \left(\dfrac{x}{\tilde{a}}\right)\theta_{a,i}(x) + 	\left(1-\dfrac{x-\tilde{b}}{\ell(t)-\tilde{b}}\right)\theta_{b,i}(x) + \theta_{c,i}(x)
	+ N\theta_{d,i}(x),
	\end{equation}
	where $N>0$ is a constant. We still have $\eta_i^*\in C^\infty(\overline{Q_{\ell}})$ satisfying \ref{F2}. Since
	$$
	\eta_i^* \leq \left(\dfrac{a_i'}{\tilde{a}}\right) + \left(1-\dfrac{b_i'-\tilde{b}}{\ell^{*}-\tilde{b}}\right) + 1
	\mbox{\ in\ } \left[0,a_i'- \dfrac{1}{9}(a_i'-a_i)\right]\cup\left[b_i'+\dfrac{1}{9}(b_i-b_i'),\ell(t)\right],\quad \forall t\in[0,T],
	$$
	$$
	\eta_i^* \leq N + 1
	\mbox{\ in\ } \left[a_i'- \dfrac{1}{9}(a_i'-a_i),a_i'\right]\cup\left[b_i',b_i'+\dfrac{1}{9}(b_i-b_i')\right],\quad \forall t\in[0,T],
	$$
	$$
	\eta_i^* = N + 1	\mbox{\ in\ } \left[a_i',b_i'\right],\quad \forall t\in[0,T],
	$$
	then, taking
    $$
    N=\left(\dfrac{a_i'}{\tilde{a}}\right) + \left(1-\dfrac{b_i'-\tilde{b}}{\ell^{*}-\tilde{b}}\right)>0,
    $$
    we obtain  $\|\eta^{*}_1\|_\infty = \|\eta^{*}_2\|_\infty = N+1$. This concludes the proof.
    
\qed

\begin{remark}
    A $C^2(\mathbb{R})$ cut-off function $\theta=\theta(x)$ satisfying $0\leq\theta\leq1$, $\theta(x)=1$ in $[a,b]$, and $\theta=0$ in $[c,d]$, with $a<b<c<d$, can be easily found by solving the problem
    \begin{equation}\left\{
        \begin{array}{lll}
         p(x)=a_0+a_1x+a_2x^2+a_3x^3+a_4x^4+a_5x^5,\\
         p(b)=1,\quad p'(b)=p''(b)=p(c)=p'(c)=p''(c)=0,
    \end{array}\right.
    \end{equation}
    and setting $\theta(x)=1$ in $(-\infty,b]$, $\theta(x)=0$ in $[c,+\infty)$, and $\theta=p$ in $(b,c)$. One example of a suitable polynomial $p(x)$ is the decreasing function
$$
p(x)
= 1
- 10\!\left(\frac{x-b}{c-b}\right)^{3}
+ 15\!\left(\frac{x-b}{c-b}\right)^{4}
- 6\!\left(\frac{x-b}{c-b}\right)^{5},
\quad x\in(b,c).
$$
    \end{remark}

\renewcommand{\thesection}{B}
\section{Analysis of optimal system in non-cylindrical domains and regularity property}\label{AnalisisofOptimalsystem}

In this section, we present some technical results that were previously used in the analysis of the linear optimal system. These are well-known results on the local and global regularity of solutions to linear parabolic equations (see, for instance, \cite{Bodart11012004,gonzalez2006controllability, ladyzhenskai1968linear, wu2006elliptic}). Although these results are standard, we include them for completeness.

We first introduce the notation used throughout this section.

Let $\mathcal{V}$ be an open interval of $\mathbb{R}$. For $p\in[1,\infty)$, we consider the Banach space
\begin{equation*}
X^{p}(0,T,\mathcal{V})=L^{p}(0,T;W^{2,p}(\mathcal{V}))\cap W^{1,p}(0,T;L^{p}(\mathcal{V}));
\end{equation*} 
and henceforth, $C$ denotes a positive constant such that $C=C(\sigma,\|c\|_{\infty},\|d\|_{\infty},\mu_{1},\mu_{2})$. 
We begin with a result on the existence, uniqueness, and regularity of solutions to the optimality system.

\begin{proposition}\label{thm-well-posedeness1} Suppose that $f \in L^{2}(\mathcal{O} \times (0,T))$, $y_{i,d}\in L^{2}(\mathcal{O}_{i,d}\times(0,T))$ and that the coefficient
$a\in L^{\infty}(Q_{\ell})$. Then, for any $y_{0}\in H^1_{0}(0,\ell_{0})$, the linear optimal system \eqref{optimalsystem} has a unique strong solution $(y,\phi^{1},\phi^{2})\in [X^{2}(0,T)]^{3}$. Moreover, there exists a constant $C>0$ such that
\begin{eqnarray*}
\|y\|_{L^{2}(0,T;L^{2}(0,\ell_{0}))}^{2}+\|y_{x}\|_{L^\infty(0,T; L^2(0,\ell_{0}))}^2 +\sum_{i=1}^{2}\|\phi^{i}\|_{L^{2}(0,T;L^{2}(0,\ell_{0}))}^{2}+\|\phi^{i}_{x}\|_{L^\infty(0,T; L^2(0,\ell_{0}))}^2\nonumber\\ \leq C\left(\|f\|_{L^2(\mathcal{O}\times(0,T))}^2 + \|y_{0}\|_{H^{1}_{0}(0,\ell_0)}^2+\sum_{i=1}^{2}\|y_{i,d}\|_{L^{2}(\mathcal{O}_{i,d}\times(0,T))}\right).    
\end{eqnarray*} 
\end{proposition}

The proof is standard, relying on a diffeomorphism that maps the noncylindrical domain to a cylindrical one and on the Faedo–Galerkin method to obtain a unique strong solution; we only sketch it below.

\subsection*{Proof of Proposition \texorpdfstring{\ref{thm-well-posedeness1}}{2}}\label{wellposednessoptimalsystem}
As a first step, we introduce a suitable diffeomorphism given by rescaling that transforms the noncylindrical domain into a cylindrical one for the optimal system. More precisely, the system \eqref{optimalsystem} can be transformed into an equivalent system on the cylindrical $Q=(0,\ell_{0})\times(0,T)$ via a coordinate transformation $\Phi_{\ell}:Q_{\ell}\mapsto Q$ and $\hat{\Phi}^{i}_{\ell}:Q_{\ell}\mapsto Q$ given by
\begin{equation}\label{diffeo}
\Phi_{\ell}(x,t)=(\xi,t)=\left(\frac{x\ell_{0}}{\ell(t)},t\right)\,\,\text{and}\,\, \hat{\Phi}^{i}_{\ell}(x,t)=(\xi,t)=\left(\frac{x\ell_{0}}{\ell(t)},T-t\right).
\end{equation}

Define
$y(x,t)=z(\xi,t), \, \phi^{1}(x,t)=\varphi^{1}(\xi,t)$ and $\phi^{2}(x,t)=\varphi^{2}(\xi,t)$. Then one easily verifies that the optimality system can be rewritten as follows:
\begin{equation}\label{optimaltransformedsystem}
\left\{\begin{aligned}
&z_{t} - b(\xi,t)z_{\xi\xi}+c(\xi,t)z_{\xi}+d(\xi,t)z=\tilde{f} \mathds{1}_{\mathcal{O}} - \sum_{i=1}^{2}\frac{1}{\mu_i} \varphi^{i} \mathds{1}_{\mathcal{O}_i} &&\text{in}&& Q, \\
&-\varphi_{t}^{i} - b(\xi,t)\varphi_{\xi\xi}+ c(\xi,t)\varphi_{\xi}+d(\xi,t)\varphi =(z-z_{i,d}) \mathds{1}_{\mathcal{O}_{i,d}}  &&\text{in}&& Q,\\
&z(0,t)=z(1,t)=0, \ \ \varphi^i(0,t) = \varphi^i(1,t) = 0 &&\text{in}&& (0,T), \\
&z(\cdot,0) =z_{0}(\xi), \ \ \varphi^i(\cdot,T)=0 &&\text{in}&& (0,\ell_{0}),
\end{aligned}
\right.
\end{equation}
where $b(\xi,t),c(\xi,t)$, $d(\xi,t)$ are bounded continuous on $Q$ such that $0<b_{0}<b(x,t)$. Since $\Phi_{\ell}$ and $\hat{\Phi}^{i}_{\ell}$ are invertible, to establish the well-posedness result for \eqref{optimalsystem}, it suffices to establish well-posedness of system \ref{optimaltransformedsystem} together with the relevant estimates.

Let $(w_{i})_{i=1}^{\infty}$ be an orthonormal basis of $H^{1}_{0}(0,\ell_{0})$. Fix $m\in\mathbb{N}^{*}$. By Carathéodory's theorem, there exist absolutely continuous functions $g_{im}=g_{im}(t)$ and $h_{im}=h_{im}(t)$ with $i\in\{1,2,\cdots,m\}$ such that 
\begin{equation*}
z_{m}(t)=\sum_{i=1}^{m}g_{im}(t)w_{i} \in H^{1}_{0}(0,\ell_{0})\,\,\, \text{ and }\,\,\, \varphi^{i}_{m}(t)=\sum_{k=1}^{m}h^{i}_{km}(t)w_{k} \in H^{1}_{0}(0,\ell_{0}),  
\end{equation*}
satisfy
\begin{equation}
\label{eq:galerkin_system}
\left\{\begin{aligned}
&(z_{m,t},\bar{w}) - b(\xi,t)(z_{m,\xi\xi},\bar{w})+c(\xi,t)(z_{\xi,m},\bar{w})\\
&+d(\xi,t)(z,\bar{w}) = (\tilde{f}\mathds{1}_{\mathcal{O}},\bar{w}) - \sum_{i=1}^{2}\frac{1}{\mu_i} (\varphi^{i}_{m} \mathds{1}_{\mathcal{O}_{i}},\bar{w}) &&\text{in}&& Q_{\ell_{0}}, \\
&(\varphi_{m,t}^i,\hat{w}^{i})-b(\xi,t)(\varphi_{m,\xi\xi}^{i},\hat{w}^{i})+c(\xi,t)(\varphi_{m}^{i},\hat{w}^{i})\\&+d(\xi,t)(\varphi_{m}^{i},\hat{w}^{i}) =((z_m -z_{i,d})\mathds{1}_{\mathcal{O}_{i,d}},\hat{w}^{i})   &&\text{in}&& Q_{\ell_{0}},\\
&z_{m}(0,t)=z_{m}(\ell_{0},t)=0, \ \ \varphi_{m}^i(0,t) = \varphi_{m}^i(\ell_{0},t) = 0 &&\text{in}&& (0,T), \\
&z_{m}(\cdot,0)\to z_{0},  \ \ \varphi_{m}^i(\cdot,0)\to\varphi^i_{0} , 
\end{aligned}
\right.
\end{equation}
for any $\bar{w},\hat{w}^{i}\in \{w_{1},w_{2},...,w_{m}\}$, where $(\cdot,\cdot)$ denotes the inner product in $L^{2}(0,\ell_{0})$. The system \eqref{eq:galerkin_system} has a solution on an interval [$0, t_{m}]$, with $t_{m}\leq T$. This solution can be extended to the whole interval $[0,T]$ as a consequence of the a priori estimates that shall be proved in the next step. 
\\
\noindent\textbf{Estimate (I)}. Taking $(\bar{w},\hat{w}^{i})=(z_{m},\varphi^{i}_{m})$ in equations $\eqref{eq:galerkin_system}{1}$ and $\eqref{eq:galerkin_system}{2}$, respectively, adding the resulting terms, integrating over $(0,t)$ and $(t,T)$ with $0\leq t < t{m} < T$, using the boundedness of the coefficients, and finally applying Gronwall's inequality, we obtain the following estimate
\begin{eqnarray*}\label{estimateenergy}
\|z_{m}(t)\|^{2}_{L^{\infty}(0,T;L^{2}(0,\ell_{0}))}+\|z_{m,\xi}\|^{2}_{L^{2}(0,T;L^{2}(0,\ell_{0}))}+\sum_{i=1}^{2}\|\varphi^{i}_{m}(t)\|^{2}_{L^{\infty}(0,T;L^{2}(0,\ell_{0}))}+\|\varphi^{i}_{m,\xi}\|^{2}_{L^{2}(0,T;L^{2}(0,\ell_{0}))}\nonumber\\ \leq C\left( \|\tilde{f}\|^{2}_{L^{2}(\mathcal{O}\times(0,T))}+\|z_{0}\|^{2}_{L^{2}(0,\ell_{0})}+\sum_{i=1}^{2}\|z_{i,d}\|^{2}_{L^{2}(\mathcal{O}_{i,d}\times(0,T))}+\|\varphi_{0}^{i}\|^{2}_{L^{2}(0,\ell_{0})}\right)
\end{eqnarray*}
\noindent\textbf{Estimate (II)}. Now, taking $w=z_{m,t}$ and $\hat{w}=\varphi^{i}_{m,t}$ in $\eqref{eq:galerkin_system}_{1}$ and $\eqref{eq:galerkin_system}_{2}$, respectively, using the resulting identities, applying Poincaré's inequality, integrating over $(0,t)$ with $0\leq t<t_{m}<T$, and using Gronwall's inequality, we obtain
\begin{eqnarray*}
\|z_{m,t}\|^{2}_{L^{2}(0,T;L^{2}(0,\ell_{0}))}+\|z_{m,\xi}\|^{2}_{L^{\infty}(0,T;L^{2}(0,\ell_{0}))}+\sum_{i=1}^{2}\|\varphi^{i}_{m,t}\|^{2}_{L^{2}(0,T;L^{2}(0,\ell_{0}))}+\|\varphi^{i}_{m,\xi}\|^{2}_{L^{\infty}(0,T;L^{2}(0,\ell_{0}))}\\  \leq C\left(\|\tilde{f}\|^{2}_{L^{2}(\mathcal{O}\times(0,T))}+\|z_{0}\|^{2}_{H_{0}^{1}(0,\ell_{0})}
+\sum_{i=1}^{2}\|z_{i,d}\|^{2}_{L^{2}(\mathcal{O}_{i,d}\times(0,T))}+\|\varphi^{i}_{0}\|^{2}_{H_{0}^{1}(0,\ell_{0})}\right)
\end{eqnarray*}

\noindent\textbf{Estimate (III)}. Finally, taking $w=-z_{m,\xi\xi}$ and $\hat{w}=- \varphi^{i}_{m,\xi\xi}$ and arguing as before, we have
\begin{eqnarray*}
\|z_{m,\xi}\|^{2}_{L^{2}(0,\ell_{0})}+\sum_{i=1}^{2}\|\varphi^{i}_{m,\xi}\|^{2}_{L^{2}(0,\ell_{0})}+\|z_{m,\xi\xi}\|^{2}_{L^{2}(0,T;L^{2}(0,\ell_{0}))}+\sum_{i=1}^{2}\|\varphi^{i}_{m,\xi\xi}\|^{2}_{L^{2}(0,T;L^{2}(0,\ell_{0}))}\\
\\ \leq C\left(\|\tilde{f}\|^{2}_{L^{2}(\mathcal{O}\times(0,T))}+\sum_{i=1}^{2}\|z_{i,d}\|^{2}_{L^{2}(\mathcal{O}_{i,d}\times(0,T))}+\|\varphi^{i}_{0}\|^{2}_{H_{0}^{1}(0,\ell_{0})}\right).    
\end{eqnarray*}
Then, from estimates (I)-(III), we deduce that $z_{m}$ is bounded in $L^{\infty}(0,T;H_{0}^{1}(0,\ell_{0}))$ and $z_{m,t}$ in $L^{2}(0,T;L^{2}(0,\ell_{0}))$. Hence, there exists a subsequence, denoted in the same way, such that
\begin{eqnarray*}
\left\{\begin{aligned}
&z_{m}\overset{*}{\rightharpoonup}z,\,\, \varphi^{i}_{m}\overset{*}{\rightharpoonup} \varphi^{i} \quad &\text{in }& L^{\infty}(0,T;H^{1}_{0}(0,\ell_{0})),\\
&z_{m,\tau}\rightharpoonup z, \,\, \varphi^{i}_{m,\tau}\rightharpoonup \varphi^{i} \quad &\text{in }& L^{2}(0,T;H^{1}_{0}(0,\ell_{0})),\\
&z_{m}\rightharpoonup z, \,\, \varphi^{i}_{m}\rightharpoonup \varphi^{i} \quad &\text{in }& L^{2}(0,T;H^{1}_{0}(0,\ell_{0})\cap H^{2}(0,\ell_{0})).
\end{aligned}
\right.    
\end{eqnarray*}
Therefore, the optimality system \eqref{optimalsystem} has a strong solution. Uniqueness follows by standard arguments for linear systems, using the bounds in estimates (I) - (III) for $z$, $z_\xi$, $\varphi^i$ and $\varphi^i_\xi$ $(i=1,2)$.
\begin{remark}
The results of Proposition \ref{thm-well-posedeness1} can also be extended to $X^{p}(0,T)$, for any $p \in [2,\infty)$, combining Theorem~2.3 of ~\cite{GIGA199172} with a bootstrap argument. More precisely, if $z_{0} \in W^{2-2/p,p}(0,\ell_{0}) \cap W^{1,p}_{0}(0,\ell_{0})$, $\tilde{f}\in L^{p}(0,T;L^{p}(\mathcal{O}))$ and $z_{i,d}\in L^{p}(0,T;L^{p}(\mathcal{O}_{i,d}))$, then $(z,\varphi^{1},\varphi^{2}) \in [X^{p}(0,T)]^{3}$. See, for instance, \cite{gonzalez2006controllability}.
\end{remark}

In addition to the global well-posedness results, one can derive further local parabolic regularity properties for the linear optimal system.

\begin{proposition}\label{thm-local-well-posedeness}
Suppose that $f \in L^{2}(\mathcal{O} \times (0,T))$, $y_{i,d}\in L^{\infty}(\mathcal{O}_{i,d}\times(0,T))$ and that the coefficient
$a \in L^{\infty}(Q_{\ell})$. Let $(y,\phi^{1},\phi^{2}) \in X^{2}(0,T;(0,\ell_{0}))$
be the unique strong solution to problem \eqref{optimalsystem}. Let $\mathcal{V} \subset (0,\ell_{0})$ be an open interval.
If $f\in L^{p}(0,T;L^{p}(\mathcal{V}))$ and $y_{i,d}\in L^{p}(0,T;L^{p}(\mathcal{V}))$ with $p \in (2,\infty)$, then for any open interval $\mathcal{V}' \subset \mathcal{V}$, the solution  $(y,\phi^{1},\phi^{2}) \in X^{p}(0,T;\mathcal{V}')$. Moreover, there exists a constant $C>0$ such that
\begin{eqnarray*}
\|y\|_{X^{p}(0,T;\mathcal{V}')}
+ \sum_{i=1}^{2}\|\phi^{i}\|_{X^{p}(0,T;\mathcal{V}')}
\leq C \Bigg(
\|f\|_{L^{p}(0,T;L^{p}(\mathcal{V}))}+ \|y\|_{X^{2}(0,T;(0,\ell_{0}))}\\+\sum_{i=1}^{2}\|y_{i,d}\mathds{1}_{\mathcal{O}_{i,d}}\|_{L^{p}(0,T;L^{p}(\mathcal{V}))} +\sum_{i=1}^{2}\|\phi^{i}\|_{X^{2}(0,T;(0,\ell_{0}))}
\Bigg).
\end{eqnarray*}
\end{proposition}

The proof of this proposition follows from the fact that, in the optimal system, the equations are not coupled in the highest-order terms and relies on the local regularity properties of the heat equation combined with a bootstrap-type argument. Since this result is not the main focus of the present paper, the proof is omitted. See Proposition~2.1 of \cite{Bodart11012004} for further details.

We recall standard function spaces arising in the regularity theory of parabolic equations and review classical results on H\"older regularity (see \cite{krylov1996lectures, ladyzhenskai1968linear,lieberman1996second}). For any $\alpha \in (0,1)$, we denote by $C^{m+\alpha,(m+\alpha)/2}(\overline{Q})$ the space of functions $u : \overline{Q} \mapsto \mathbb{R}$ such that $D_t^r D_x^s u$ is continuous in $Q$ for $2r+s\leq m+\alpha$, with $m$ a nonnegative integer, and is a separable Banach space with the norm given by
\begin{eqnarray*}
\| \cdot \|_{C^{m+\alpha,(m+\alpha)/2}(\overline{Q})} &=&
\sum_{2r + s \leq m} \| D_t^r D_x^s u \|_{\infty}\\ &\quad&+ \sum_{2r + s = m}\left(\sup_{(x,t),(x',t')\in Q}\frac{|D_{t}^{r}D_{x}^{s}u(x,t)-D_{t}^{r}D_{x}^{s}u(x',t')|}{|x-x'|^{\alpha}+|t-t'|^{\alpha/2}}\right) <\infty
\end{eqnarray*}
We have the following continuous embedding, adapted from Lemma~3.3 in~\cite{ladyzhenskai1968linear} to our notation; see~\cite{Bodart11012004} for more general versions. Let $\mathcal{V} \subset \mathbb{R}$ be an open interval whose boundary is sufficiently regular. Then,
\begin{equation}\label{immersCalpha}
X^{p}(0,T;\mathcal{V}) \hookrightarrow 
C^{1+\alpha,(1+\alpha)/2}\left(\overline{\mathcal{V}} \times [0,T]\right),
\,\, p>3,\ \alpha = 1-\dfrac{3}{p}.
\end{equation}

\subsection{Proof of Lemma \texorpdfstring{\ref{lemmmaregularity1}}{2}}\label{regularityproperty}

\begin{proof} We first apply a suitable change of variables
$\Phi_{\ell}(x,t)$ and $\hat{\Phi}^{i}_{\ell}(x,t)$ such that the region $R_{\ell}$ is transformed into a fixed region \(R_{\ell_{0}}\), where $\ell^{*}$ and $\ell_{0}$ correspond to the endpoints of the new region. 
\begin{figure}[b]
   \centering
   \begin{tikzpicture}[scale=0.7]

\begin{scope}

\fill[gray!10]
  (0,0)
  -- (4.5,0)
  .. controls (4.9,1.2) and (5.1,0.5) .. (5.5,2.0)
  .. controls (5.9,3.2) and (6.3,2.8) .. (6.5,4.0)
  -- (0,4.0)
  -- cycle;

\fill[blue!5]
  (3.3,0)
  -- (4.5,0)
  .. controls (4.9,1.2) and (5.1,0.5) .. (5.5,2.0)
  .. controls (5.9,3.2) and (6.3,2.8) .. (6.5,4.0)
  -- (3.3,4.0)
  -- cycle;

\fill[gray!10]
  (4.5,0)
  .. controls (4.9,1.2) and (5.1,0.5) .. (5.5,2.0)
  .. controls (5.9,3.2) and (6.3,2.8) .. (6.5,4.0)
  -- (6.5,0)
  -- cycle;  

\draw[thick, blue]
  (4.5,0) 
  .. controls (4.9,1.2) and (5.1,0.5) .. (5.5,2.0)
  .. controls (5.9,3.2) and (6.3,2.8) .. (6.5,4.0);
\node at (6.0,2.0) {$\ell$};

\draw[dashed, thick, red] (3.3,0) -- (3.3,4);

\draw (3.3,0) node[below] {$\ell^{*}$} node[circle,fill,inner sep=1pt] {};
\draw (4.5,0) node[below] {$\ell_0$} node[circle,fill,inner sep=1pt] {};

\draw[dashed] (6.5,4.0) -- (6.5,0);

\draw (6.5,0) node[below] {$B$} node[circle,fill,inner sep=1pt] {};

\fill[gray!5] (0.2,0) rectangle (1.0,4.0);
\fill[pattern=north east lines] (0.2,0) rectangle (1.0,4.0);
\fill[gray!5] (0.6,2.0) circle (0.35);
\node at (0.6,2.0) {$\mathcal{O}_0$};

\fill[gray!5] (1.2,0) rectangle (2.0,4.0);
\fill[pattern=north east lines] (1.2,0) rectangle (2.0,4.0);
\fill[gray!5] (1.6,2.0) circle (0.35);
\node at (1.6,2.0) {$\mathcal{O}_1$};

\fill[gray!5] (2.2,0) rectangle (3.0,4.0);
\fill[pattern=north east lines] (2.2,0) rectangle (3.0,4.0);
\fill[gray!5] (2.6,2.0) circle (0.35);
\node at (2.6,2.0) {$\mathcal{O}_2$};

\node at (4.6,2.5) {$R_{\ell}$};

\draw[->] (0,0) -- (7.2,0) node[right] {$x$};
\draw[->] (0,0) -- (0,4.2) node[above] {$t$};

\draw (-0.1,4) node[left] {$T$};
\draw (0.1,-0.1) node[left] {$Q_{\ell}$};


\end{scope}

\begin{scope}[xshift=7.0]
\node at (3,-1.0) {\text{(a) Region $R_{\ell}$}};
\end{scope}

\begin{scope}[xshift=9.5cm]
\fill[gray!10]
  (0,0)
  -- (4.5,0)
  -- (4.5,4.0)
  -- (0,4.0)
  -- cycle;
\fill[gray!10]
  (4.5,0)
  -- (4.5,4.0)
  -- (6.5,4.0)
  -- (6.5,0)
  -- cycle;
\fill[blue!5]
  (3.3,0)
  -- (3.3,4.0)
  -- (5.5,4.0)
  -- (5.5,0)
  -- cycle;

\draw[thick, blue] (5.5,0) -- (5.5,4.0);
; 

\draw[dashed, thick, red] (3.3,0) -- (3.3,4);

\draw (3.3,0) node[below] {$\ell^{*}$} node[circle,fill,inner sep=1pt] {};
\draw (5.5,0) node[below] {$\ell_0$} node[circle,fill,inner sep=1pt] {};

\draw[dashed] (6.5,4.0) -- (6.5,0);

\draw (6.5,0) node[below] {$B$} node[circle,fill,inner sep=1pt] {};

\fill[gray!5] (0.2,0) rectangle (1.0,4.0);
\fill[pattern=north east lines] (0.2,0) rectangle (1.0,4.0);
\fill[gray!5] (0.6,2.0) circle (0.35);
\node at (0.6,2.0) {$\mathcal{O}_0$};

\fill[gray!5] (1.2,0) rectangle (2.0,4.0);
\fill[pattern=north east lines] (1.2,0) rectangle (2.0,4.0);
\fill[gray!5] (1.6,2.0) circle (0.35);
\node at (1.6,2.0) {$\mathcal{O}_1$};

\fill[gray!5] (2.2,0) rectangle (3.0,4.0);
\fill[pattern=north east lines] (2.2,0) rectangle (3.0,4.0);
\fill[gray!5] (2.6,2.0) circle (0.35);
\node at (2.6,2.0) {$\mathcal{O}_2$};

\node at (4.6,2.5) {$R^{*}_{\ell}$};

\draw[->] (0,0) -- (7.2,0) node[right] {$\xi$};
\draw[->] (0,0) -- (0,4.2) node[above] {$t$};

\draw (0,4) node[left] {$T$};

\draw (0.1,-0.1) node[left] {$Q_{\ell_{0}}$};

\node at (3,-1.0) {\text{(b) Region $R_{\ell_{0}}$}};
\end{scope}
\end{tikzpicture}
   \label{fig:changeofvariable1}
\end{figure}
Due to the regularity of $z_{0}$, we can introduce a shift function $u$ for the initial data such that 
$u,g^{i} \in X^{4}(0,T;(0,\ell_{0}))$, with $u(0,t)=u(\ell_{0},t)=0$, $g^{i}(0,t)=g^{i}(\ell_{0},t)=0$ for $t \in (0,T)$, $i=1,2$ and
$u(\xi,0)=z_{0}(\xi)$ for $\xi \in (0,\ell_{0})$. Consequently, the states $(z,\varphi^{1},\varphi^{2})$ can be written in the form $z = u + p$ and $\varphi^{i} = g^{i} + q^{i}$, where $(p,q^{1},q^{2}) \in [X^{2}(0,T;(0,\ell_{0}))]^{3}$

is the unique strong solution of the following problem:
\begin{equation}\label{splitsolution}
		\left\{
  \begin{array}{llllll}
	\displaystyle p_{t}-b(\xi,t)p_{\xi\xi}+c(\xi,t)p_{\xi}+d(\xi,t)p=F(\xi,t)-\sum_{i=1}^{2}\dfrac{1}{\mu_{i}}q^{i}\mathds{1}_{\mathcal{O}_{i}} &&\text{in}&& Q_{\ell_{0}},
\\[4pt]\noalign{\smallskip}
	-q^{i}_{t}-b(\xi,t)q^{i}_{\xi\xi}+c(\xi,t)q^{i}_{\xi}+d(\xi,t)q^{i}=G(\xi,t)+(p-z_{i,d})\mathds{1}_{\mathcal{O}_{i,d}} &&\text{in}&& Q_{\ell_{0}},
\\[4pt]\noalign{\smallskip}
    p(0,t)=p(\ell_{0},t)=0, \quad q^{i}(0,t)=q^{i}(\ell_{0},t)=0 &&\text{in} && (0,T),
\\[4pt]\noalign{\smallskip}
	p(\xi,0)=0 &&\text{in} && (0,\ell_{0}),
		\end{array}
		\right.
\end{equation} 
where 
\begin{equation*}
F(\xi,t)=-u_{t}+b(\xi,t)u_{\xi\xi}-c(\xi,t)u_{\xi}-d(\xi,t)u+\tilde{f}\mathds{1}_{\mathcal{O}}-\sum_{i=1}^{2}\frac{1}{\mu_{i}}g^{i}\mathds{1}_{\mathcal{O}_{i}} 
\end{equation*}
and
\begin{equation*}
G^{i}(\xi,t)=g^{i}_{t}+b(\xi,t)g^{i}_{\xi\xi}-c(\xi,t)g^{i}_{\xi}-d(\xi,t)g^{i}+u
\end{equation*}

Let $\delta > 0$ be sufficiently small such that $\hat{b} < \ell^{*} - \delta$ and $\ell^{*} + \delta < \ell_{0}$.  Since $\mathcal{O}_{i}\cap[\ell^{*}-\delta,\ell^{*}+\delta]=\emptyset$ for each $i=0,1,2$ and $F,G^{1},G^{2}\in L^{4}(0,T;L^{4}(\ell^{*}-\delta,\ell^{*}+\delta))$. Proposition \ref{thm-local-well-posedeness} yields
\begin{equation*}
(p,q^{1},q^{2})\in [X^{4}(0,T;((\ell^{*}-\delta)/2,(\ell^{*}+\delta)/2)]^{3},
\end{equation*}
and there exists a constant $C>0$ such that
\begin{eqnarray*}
\|p\|_{X^{4}(0,T;(\frac{\ell^{*}-\delta}{2},\frac{\ell^{*}+\delta}{2}))}+\sum_{i=1}^{2}\|q^{i}\|_{X^{4}(0,T;(\frac{\ell^{*}-\delta}{2},\frac{\ell^{*}+\delta}{2}))}\leq C\left(\|F\|_{L^{4}(0,T;L^{4}(\ell^{*}-\delta,\ell^{*}+\delta))}\phantom{\sum_{i=1}^{2}}\right. \\ 
+\sum_{i=1}^{2}\|G^{i}\|_{L^{4}(0,T;L^{4}(\ell^{*}
-\delta,\ell^{*}+\delta))}+\sum_{i=1}^{2}\|z_{i,d}\mathds{1}_{\mathcal{O}_{i,d}}\|_{L^{4}(0,T;L^{4}(\ell^{*}-\delta,\ell^{*}+\delta))}\\+\|p\|_{X^{2}(0,T;(0,\ell_{0}))}
\left.+\sum_{i=1}^{2}\|q^{i}\|_{X^{2}(0,T;(0,\ell_{0}))}\right).    
\end{eqnarray*}
Next, using standard parabolic energy estimates and regularity of the control, we obtain
\begin{equation*}
\|p\|_{{X^{4}(0,T;(\frac{\ell^{*} - \delta}{2}, \frac{\ell^{*} + \delta}{2}))}}+\sum_{i=1}^{2}\|q^{i}\|_{{X^{4}(0,T;(\frac{\ell^{*}-\delta}{2}, \frac{\ell^{*} + \delta}{2}))}} \leq 
        C\left(\|p_0\|_{W_{0}^{1,4}(0,\ell_{0})} + \sum_{i=1}^{2}\|z_{i,d}\|_{L^{4}(0,T;L^{4}(\mathcal{O}_{i,d}))}\right)   
\end{equation*}
Finally, using this inequality, the regularity of the trace $p(\ell_{0},\cdot),q^{1}(\ell_{0},\cdot),q^{2}(\ell_{0},\cdot)$, the fact that $(p,q^{1},q^{2})$ is a strong solution to \eqref{splitsolution} and ~\cite[Propositions $9.2.3$ and $9.2.5$]{wu2006elliptic}, we conclude that $(p,q^{1},q^{2})\in X^{4}(0,T;(\ell^{*},\ell_0))$ and, moreover,
        \begin{equation}\label{estimate4.7}
        \|p\|_{X^{4}(0,T;R_{\ell_{0}})}+\sum_{i=1}^{2}\|q^{i}\|_{X^{4}(0,T;R_{\ell_{0}})} \leq C\left(\|p_0\|_{W_0^{1,4}(0,\ell_{0})}+\sum^{2}_{i=1}\|z_{i,d}\|_{L^{4}(0,T;L^{4}(\mathcal{O}_{i,d}))}\right),
        \end{equation}
 for a new $C > 0$.
 Then, estimate in \eqref{lemmmaregularity1} is an immediate consequence of \eqref{estimate4.7} and embedding \ref{immersCalpha}, yielding 
\begin{equation*}
 \|z\|_{C^{1+\alpha,\frac{(1+\alpha)}{2}}(\overline{R_{\ell}})}+\sum_{i=1}^{2}\|\varphi^{i}\|_{C^{1+\alpha,\frac{(1+\alpha)}{2}}(\overline{R_{\ell}})}\leq C\left(\|z_{0}\|_{W^{1,4}_{0}(0,\ell_{0})}+\sum_{i=1}^{2}\|z_{i,d}\|_{L^{4}(0,T;L^{4}(\mathcal{O}_{i,d}))}\right).
\end{equation*}
Since $\rho(t)\geq1$ for all $t\in[0,T]$, we can conclude that 
\begin{equation*}
 \|y\|_{C^{1+\alpha,\frac{(1+\alpha)}{2}}(\overline{R_{\ell}})}+\sum_{i=1}^{2}\|\phi^{i}\|_{C^{1+\alpha,\frac{(1+\alpha)}{2}}(\overline{R_{\ell}})}\leq C\left(\|y_{0}\|_{W^{1,4}_{0}(0,\ell_{0})}+\sum_{i=1}^{2}\|\rho y_{i,d}\|_{L^{4}(0,T;L^{4}(\mathcal{O}_{i,d}))}\right).
\end{equation*}

\end{proof}

\bibliographystyle{abbrv}
\bibliography{bib}

@article{costa2023controllability,
  title={On the controllability of a free-boundary problem for 1D heat equation with local and nonlocal nonlinearities},
  author={Costa, V and L{\'\i}maco, J and Lopes, AR and Prouv{\'e}e, L},
  journal={Applied Mathematics \& Optimization},
  volume={87},
  number={1},
  pages={11},
  year={2023},
  publisher={Springer}
}

@article{F-E-M2,
title = {New results on the Stackelberg–Nash exact control of linear parabolic equations},
journal = {Systems \& Control Letters},
volume = {104},
pages = {78-85},
year = {2017},
issn = {0167-6911},
doi = {https://doi.org/10.1016/j.sysconle.2017.03.009},
url = {https://www.sciencedirect.com/science/article/pii/S0167691117300622},
author = {Araruna, F{\'a}gner and E. Fernández-Cara and S. Guerrero and M.C. Santos}
}

@article{F-E-M,
	author = {Araruna, F{\'a}gner and E. Fernández-Cara and M.C. Santos},
	title = {Stackelberg-Nash exact controllability for linear and
          semilinear parabolic equations},
	DOI= "10.1051/cocv/2014052",
	url= "https://doi.org/10.1051/cocv/2014052",
	journal = {ESAIM: COCV},
	year = 2015,
	volume = 21,
	number = 3,
	pages = "835-856",
	month = "",
}

@inproceedings{Di-L,
  title={On the approximate controllability of Stackelberg-Nash strategies},
  author={D{\'\i}az, JI and Lions, JL},
  booktitle={Ocean Circulation and Pollution Control—A Mathematical and Numerical Investigation: A Diderot Mathematical Forum},
  pages={17-27},
  year={2004},
  organization={Springer}
}

@article{Lions1,
  title={Hierarchic control},
  author={Lions, JL},
  journal={Proceedings Mathematical Sciences},
  volume={104},
  pages={295--304},
  year={1994},
  publisher={Springer}
}

@book{stackelberg,
  title={Marktform und Gleichgewicht},
  author={von Stackelberg, H.},
  lccn={36009325},
  series={Die Handelsblatt-Bibliothek ``Klassiker der National{\"o}konomie"},
  url={https://books.google.com.br/books?id=wihBAAAAIAAJ},
  year={1934},
  publisher={J. Springer}
}

@article{Luz1,
  title={New results concerning the hierarchical control of linear and semilinear parabolic equations},
  author={Calsavara, Bianca MR and Fern{\'a}ndez-Cara, Enrique and de Teresa, Luz and Villa, Jos{\'e} Antonio},
  journal={ESAIM: Control, Optimisation and Calculus of Variations},
  volume={28},
  pages={14},
  year={2022},
  publisher={EDP Sciences}
}

@incollection{nash,
  title={Non-cooperative games},
  author={Nash, John F},
  booktitle={The Foundations of Price Theory Vol 4},
  pages={329--340},
  year={2024},
  publisher={Routledge}
}

@article{araruna2020hierarchical,
  title={Hierarchical exact controllability of semilinear parabolic equations with distributed and boundary controls},
  author={Araruna, FD and Fern{\'a}ndez-Cara, E and Da Silva, LC},
  journal={Communications in Contemporary Mathematics},
  volume={22},
  number={07},
  pages={1950034},
  year={2020},
  publisher={World Scientific}
}

@article{araruna2018stackelberg,
  title={Stackelberg--Nash null controllability for some linear and semilinear degenerate parabolic equations},
  author={Araruna, F{\'a}gner D and Ara{\'u}jo, BSV and Fern{\'a}ndez-Cara, Enrique},
  journal={Mathematics of Control, Signals, and Systems},
  volume={30},
  pages={1--31},
  year={2018},
  publisher={Springer}
}

@article{hernandez2016hierarchic,
  title={Hierarchic control for a coupled parabolic system},
  author={Hernandez-Santamaria, Victor and de Teresa, Luz and Poznyak, Alexander},
  journal={Portugaliae Mathematica},
  volume={73},
  number={2},
  pages={115--137},
  year={2016}
}

@article{araruna2024bi,
  title={Bi-objective and hierarchical control for the Burgers equation},
  author={Araruna, FD and Fern{\'a}ndez-Cara, E and da Silva, LC},
  journal={Journal of Evolution Equations},
  volume={24},
  number={2},
  pages={30},
  year={2024},
  publisher={Springer}
}

@article{lions1994some,
  title={Some remarks on Stackelberg’s optimization},
  author={Lions, Jacques-Louis},
  journal={Mathematical Models and Methods in Applied Sciences},
  volume={4},
  number={04},
  pages={477--487},
  year={1994},
  publisher={World Scientific}
}

@article{LLW13,
  title={The free boundary problem describing information diffusion in online social networks},
  author={Lei, Chengxia and Lin, Zhigui and Wang, Haiyan},
  journal={Journal of Differential Equations},
  volume={254},
  number={3},
  pages={1326--1341},
  year={2013},
  publisher={Elsevier}
}

@article{FR99,
  title={Analysis of a mathematical model for the growth of tumors},
  author={Friedman, Avner and Reitich, Fernando},
  journal={Journal of mathematical biology},
  volume={38},
  pages={262--284},
  year={1999},
  publisher={Springer}
}

@book{3,
  title={Partial differential equations of parabolic type},
  author={Friedman, Avner},
  year={2008},
  publisher={Courier Dover Publications}
}

@article{4,
  title={Variational principles and free-boundary problems},
  author={Friedman, Avner},
  journal={(No Title)},
  year={1982}
}

@book{5,
  title={Water waves and ship hydrodynamics: An introduction},
  author={Hermans, AJ},
  year={2010},
  publisher={Springer Science \& Business Media}
}

@book{6,
  title={Water waves: The mathematical theory with applications},
  author={Stoker, James Johnston},
  year={2019},
  publisher={Courier Dover Publications}
}

@book{7,
  title={Computational Modelling of Free and Moving Boundary Problems: Vol 1, Fluid Flow: Proceedings of the First International Conference held 2-4 July 1991, Southhampton, UK.},
  author={Brebbia, Carlos Alberto and Wrobel, Luiz Carlos},
  year={2016},
  publisher={De Gruyter}
}

@inproceedings{11,
  title={Some properties of the interface for a gas flow in porous media},
  author={Aronson, DG},
  booktitle={Proceedings of the Montecatini Symposium on Free Boundary Problems, Pitman, New York},
  year={1983}
}

@book{13,
  title={The porous medium equation: mathematical theory},
  author={V{\'a}zquez, Juan Luis},
  year={2007},
  publisher={Oxford university press}
}

@book{14,
  title={Tutorials in mathematical biosciences III: cell cycle, proliferation, and cancer},
  author={Friedman, Avner and Aguda, Baltazar D},
  year={2006},
  publisher={Springer}
}

@article{15,
  title={PDE problems arising in mathematical biology},
  author={Friedman, Avner},
  journal={Networks and heterogeneous media},
  volume={7},
  number={4},
  pages={691--703},
  year={2012},
  publisher={Networks and Heterogeneous Media}
}

@article{hernandez2018some,
  title={Some remarks on the hierarchic control for coupled parabolic PDEs},
  author={Hern{\'a}ndez-Santamar{\'\i}a, V{\'\i}ctor and de Teresa, Luz},
  journal={Recent Advances in PDEs: Analysis, Numerics and Control: In Honor of Prof. Fern{\'a}ndez-Cara's 60th Birthday},
  pages={117--137},
  year={2018},
  publisher={Springer}
}

@article{nina2021stackelberg,
  title={Stackelberg--Nash controllability for N-dimensional nonlinear parabolic partial differential equations},
  author={Nina-Huaman, Dany and L{\'\i}maco, J},
  journal={Applied Mathematics \& Optimization},
  volume={84},
  pages={1401--1452},
  year={2021},
  publisher={Springer}
}

@article{djomegne2023stackelberg,
  title={Stackelberg--Nash null controllability for a non linear coupled degenerate parabolic equations},
  author={Djomegne, Landry and Kenne, Cyrille and Dorville, Ren{\'e} and Zongo, Pascal},
  journal={Applied Mathematics \& Optimization},
  volume={87},
  number={2},
  pages={18},
  year={2023},
  publisher={Springer}
}

@article{carreno2019stackelberg,
  title={Stackelberg--Nash exact controllability for the Kuramoto--Sivashinsky equation},
  author={Carreno, N and Santos, MC},
  journal={Journal of Differential Equations},
  volume={266},
  number={9},
  pages={6068--6108},
  year={2019},
  publisher={Elsevier}
}

@article{carreno2023stackelberg,
  title={Stackelberg-Nash exact controllability for the Kuramoto-Sivashinsky equation with boundary and distributed controls},
  author={Carreno, Nicol{\'a}s and Santos, Maur{\'\i}cio C},
  journal={Journal of Differential Equations},
  volume={343},
  pages={1--63},
  year={2023},
  publisher={Elsevier}
}

@article{fernandez2016controllability,
  title={On the controllability of a free-boundary problem for the 1D heat equation},
  author={Fern{\'a}ndez-Cara, Enrique and Limaco, Juan and de Menezes, SB},
  journal={Systems \& Control Letters},
  volume={87},
  pages={29--35},
  year={2016},
  publisher={Elsevier}
}

@article{fernandez2019local,
  title={Local null controllability of a 1D Stefan problem},
  author={Fern{\'a}ndez-Cara, Enrique and Hern{\'a}ndez, Freddy and L{\'\i}maco, J},
  journal={Bulletin of the Brazilian Mathematical Society, New Series},
  volume={50},
  pages={745--769},
  year={2019},
  publisher={Springer}
}

@article{fernandez2017local,
  title={Local null controllability of a free-boundary problem for the semilinear 1D heat equation},
  author={Fern{\'a}ndez-Cara, Enrique and de Sousa, Ivaldo Tributino},
  journal={Bulletin of the Brazilian Mathematical Society, New Series},
  volume={48},
  number={2},
  pages={303--315},
  year={2017},
  publisher={Springer}
}

@article{fernandez2017localB,
  title={Local null controllability of a free-boundary problem for the viscous Burgers equation},
  author={Fern{\'a}ndez-Cara, E and De Sousa, IT},
  journal={SeMA Journal},
  volume={74},
  pages={411--427},
  year={2017},
  publisher={Springer}
}

@article{geshkovski2021controllability,
  title={Controllability of one-dimensional viscous free boundary flows},
  author={Geshkovski, Borjan and Zuazua, Enrique},
  journal={SIAM Journal on Control and Optimization},
  volume={59},
  number={3},
  pages={1830--1850},
  year={2021},
  publisher={SIAM}
}

@article{barcena2023exact,
  title={Exact controllability to the trajectories of the one-phase Stefan problem},
  author={B{\'a}rcena-Petisco, Jon Asier and Fern{\'a}ndez-Cara, Enrique and Souza, Diego A},
  journal={Journal of Differential Equations},
  volume={376},
  pages={126--153},
  year={2023},
  publisher={Elsevier}
}

@article{wang2022local,
  title={Local null controllability of a free-boundary problem for the quasi-linear 1D parabolic equation},
  author={Wang, Lili and Lan, Yuzhen and Lei, Peidong},
  journal={Journal of Mathematical Analysis and Applications},
  volume={506},
  number={2},
  pages={125676},
  year={2022},
  publisher={Elsevier}
}

@article{wang2023null,
  title={Null controllability of a 1D Stefan problem for the heat equation governed by a multiplicative control},
  author={Wang, Lili and Lei, Peidong and Wu, Qingzhe},
  journal={Systems \& Control Letters},
  volume={171},
  pages={105417},
  year={2023},
  publisher={Elsevier}
}

@article{wang2022insensitizing,
  title={Insensitizing controls of a 1d stefan problem for the semilinear heat equation},
  author={Wang, Lili and Lei, Peidong and Wu, Qingzhe},
  journal={Bulletin of the Brazilian Mathematical Society, New Series},
  volume={53},
  number={4},
  pages={1351--1375},
  year={2022},
  publisher={Springer}
}

@article{LIMACO2026104513,
title = {Hierarchical null controllability of a degenerate parabolic equation with nonlocal coefficient},
journal = {Nonlinear Analysis: Real World Applications},
volume = {89},
pages = {104513},
year = {2026},
issn = {1468-1218},
doi = {https://doi.org/10.1016/j.nonrwa.2025.104513},
url = {https://www.sciencedirect.com/science/article/pii/S1468121825001956},
author = {Juan Límaco and  {João Carlos Barreira} and Suerlan Silva and Luis P. Yapu}
}

@article{araujo2022remarks,
  title={Remarks on the control of two-phase Stefan free-boundary problems},
  author={Ara{\'u}jo, Raul KC and Fern{\'a}ndez-Cara, Enrique and L{\'\i}maco, Juan and Souza, Diego A},
  journal={SIAM Journal on Control and Optimization},
  volume={60},
  number={5},
  pages={3078--3099},
  year={2022},
  publisher={SIAM}
}

@article{demarque2018local,
  title={Local null controllability of one-phase Stefan problems in 2D star-shaped domains},
  author={Demarque, Reginaldo and Fern{\'a}ndez-Cara, Enrique},
  journal={Journal of Evolution Equations},
  volume={18},
  number={1},
  pages={245--261},
  year={2018},
  publisher={Springer}
}

@book{ladyzhenskai1968linear,
  title={Linear and Quasi-linear Equations of Parabolic Type},
  author={Ladyzhenskaia, O.A. and Solonnikov, V.A. and Ural'tseva, N.N.},
  isbn={9780821886533},
  lccn={68019440},
  series={American Mathematical Society, translations of mathematical monographs},
  url={https://books.google.com.br/books?id=dolUcRSDPgkC},
  year={1968},
  publisher={American Mathematical Society}
}

@book{lieberman1996second,
  title={Second order parabolic differential equations},
  author={Lieberman, Gary M},
  year={1996},
  publisher={World scientific}
}

@book{krylov1996lectures,
  title={Lectures on elliptic and parabolic equations in Holder spaces},
  author={Krylov, Nikola{\u\i} Vladimirovich},
  volume={12},
  year={1996},
  publisher={American Mathematical Soc.}
}

@book{wu2006elliptic,
  title={Elliptic and parabolic equations},
  author={Wu, Zhuoqun and Yin, Jingxue and Wang, Chunpeng},
  year={2006},
  publisher={World Scientific Publishing Company}
}

@article{Bodart11012004,
author = {O. Bodart and M. González-Burgos and R. Pérez-García},
title = {Existence of Insensitizing Controls for a Semilinear Heat Equation with a Superlinear Nonlinearity},
journal = {Communications in Partial Differential Equations},
volume = {29},
number = {7-8},
pages = {1017--1050},
year = {2004},
publisher = {Taylor \& Francis},
doi = {10.1081/PDE-200033749},
URL = {https://doi.org/10.1081/PDE-200033749},
eprint = {https://doi.org/10.1081/PDE-200033749}
}

@article{gonzalez2006controllability,
  title={Controllability results for some nonlinear coupled parabolic systems by one control force},
  author={Gonz{\'a}lez-Burgos, Manuel and P{\'e}rez-Garc{\'\i}a, Rosario},
  journal={Asymptotic Analysis},
  volume={46},
  number={2},
  pages={123--162},
  year={2006},
  publisher={SAGE Publications Sage UK: London, England}
}

@article{GIGA199172,
title = {Abstract Lp estimates for the Cauchy problem with applications to the Navier-Stokes equations in exterior domains},
journal = {Journal of Functional Analysis},
volume = {102},
number = {1},
pages = {72-94},
year = {1991},
issn = {0022-1236},
doi = {https://doi.org/10.1016/0022-1236(91)90136-S},
url = {https://www.sciencedirect.com/science/article/pii/002212369190136S},
author = {Yoshikazu Giga and Hermann Sohr}
}

@article{Djomegne18112025,
author = {Landry Djomegne and Cyrille Kenne and René Dorville and Pascal Zongo},
title = {Hierarchical null controllability of a semilinear degenerate parabolic equation with a gradient term},
journal = {Optimization},
volume = {74},
number = {15},
pages = {4007--4047},
year = {2025},
publisher = {Taylor \& Francis},
doi = {10.1080/02331934.2024.2394608},
URL = {https://doi.org/10.1080/02331934.2024.2394608},
eprint = {https://doi.org/10.1080/02331934.2024.2394608}

}

@article{imanuvilov2003carleman,
  title={Carleman inequalities for parabolic equations in Sobolev spaces of negative order and exact controllability for semilinear parabolic equations},
  author={Imanuvilov, Oleg Yu and Yamamoto, Masahiro},
  journal={Publications of the Research Institute for Mathematical Sciences},
  volume={39},
  number={2},
  pages={227--274},
  year={2003},
  publisher={Research Institute forMathematical Sciences}
}

\end{document}